\documentclass[preprint,11pt,3p]{elsarticle}
\usepackage{amsmath,amsthm,amsfonts,amssymb,arydshln}
\usepackage{xcolor}
\usepackage{algorithm}
\usepackage{multirow}
\usepackage{mathdots}
\usepackage{multicol,graphicx}
\usepackage{mathrsfs}
\usepackage{booktabs}
\newcommand{\qprod}{\mathbin{\ast}}
\usepackage{comment}
\usepackage{subcaption}
\usepackage{threeparttable}
\usepackage{lipsum}
\usepackage{graphicx}
\usepackage{algpseudocode}
\usepackage{tikz-cd}
\algnewcommand{\Input}{\item[\textbf{Input:}]}
\algnewcommand{\Output}{\item[\textbf{Output:}]}

\usepackage{mathrsfs}
\usepackage{sectsty}
\definecolor{astral}{RGB}{46,116,181}
\sectionfont{\color{astral}}
\newtheorem{theorem}{Theorem}[section]
\newtheorem{lemma}[theorem]{Lemma}

\newtheorem{definition}[theorem]{Definition}
\newtheorem{example}[theorem]{Example}
\newtheorem{remark}[theorem]{Remark}
\usepackage{hyperref}
\usepackage{scalerel}
\usepackage{stackengine,wasysym}

\usepackage{tikz}
\pgfdeclarelayer{bg} 
\pgfsetlayers{bg,main}

\usepackage{tikz,xcolor}
\usetikzlibrary{positioning}
\definecolor{lime}{HTML}{A6CE39}
\definecolor{lightblue}{rgb}{0.0, 0.0, 0.5}
\DeclareRobustCommand{\orcidicon}{%
	\begin{tikzpicture}
	\draw[lime, fill=lime] (0,0)
	circle [radius=0.16]
	node[white] {{\fontfamily{qag}\selectfont \tiny ID}};
	\draw[white, fill=white] (-0.0625,0.095)
	circle [radius=0.007];
	\end{tikzpicture}
	\hspace{-2mm}
}
\foreach \x in {A, ..., Z}{%
	\expandafter\xdef\csname orcid\x\endcsname{\noexpand\href{https://orcid.org/\csname orcidauthor\x\endcsname}{\noexpand\orcidicon}}
}

\begin{document}
\begin{frontmatter}
\title{
A Fast Algorithm for Quaternion Tensor Singular Value Decomposition under the QT-Product with Applications to Color Video Processing}
\author{Daochang Zhang$^a$, Yue Zhao$^{b,*}$, Jingqian Li$^c$}
\address{
$^{a}$ College of Mathematics, Changchun Normal University, Changchun 130032, P.R. China.
\\\textit{E-mail}: \texttt{daochangzhang@126.com}\\

$^{b}$ Department of Math \& Stat, Georgia State University, Atlanta, GA 30303, USA.
\\\textit{E-mail}: \texttt{yuezhao0303@163.com}\\

$^{c}$ School of Sciences, Tianjin University of Technology, Tianjin, P.R. China.
\\\textit{E-mail}: \texttt{JingqianLi85@163.com}
}

\begin{abstract}
Based on the complex representation of quaternion matrices, this paper revisits the quaternion tensor singular value decomposition (QT-SVD) under the QT-product and establishes the corresponding $z$-block circulant matrix decomposition. We develop an efficient implementation of the QT-SVD. Numerical experiments show that this implementation is
substantially faster than QT-SVD based on the quaternion matrix SVD approach \cite[Theorem~2.3]{AMLzhang2022} while maintaining
high numerical accuracy. We further apply the QT-SVD to color video encryption and reconstruction. Experiments on five DAVIS video sequences show that the QT-SVD achieves higher reconstruction quality and shorter decomposition times than the quaternion T-SVD \cite{tsvdarxiv} under the same processing procedure. Both methods exhibit comparable ciphertext statistics and differential sensitivity, while the QT-SVD introduces less additional distortion under the tested input perturbations.
\end{abstract}


\end{frontmatter}

\section{Introduction}
Quaternion algebra, introduced by Hamilton in 1843, is a four-dimensional associative and noncommutative extension of the complex
numbers. A quaternion can be written as
$q=a_0+a_1\mathbf{i}+a_2\mathbf{j}+a_3\mathbf{k},~ a_\ell\in\mathbb R\ (\ell=0,1,2,3)$.
This representation provides a unified description of multichannel data
and has been widely used in image processing, signal processing, and computer vision~\cite{MFBC2023,CHOU,SHOE,KUNZE}. In particular, the RGB channels of a color image can be encoded by the imaginary components of a
quaternion matrix, making QSVD a useful tool for color image processing \cite{Chenqsvd}.

For higher-order data such as video, tensor representations provide a
natural means of retaining correlations across multiple modes.
Classical tensor decompositions include the CP
decomposition~\cite{CRRROL1970}, Tucker decomposition~\cite{TUCKER1966},
HOSVD~\cite{Lathauwer2000}, and t-SVD~\cite{LAAKM2011}.
Quaternion tensor SVDs have subsequently been developed from different
algebraic viewpoints. Zheng et al.~\cite{ZMMAMLD} studied a quaternion
t-SVD using octonions, with the analysis mainly restricted to tensors
symmetric in the third mode. \cite{PanNg2024,tsvdarxiv} developed the quaternion tensor SVD under the T-product and computational algorithm. Qin et al.~\cite{AMLzhang2022} introduced the QT-product and established
the associated QT-SVD.

Although the QT-SVD provides a theoretical framework for quaternion tensor decomposition, its computational efficiency and performance in color video processing merit further
investigation. In particular, the use of direct quaternion matrix SVD may lead to substantial computational costs for large-scale tensor data.

Motivated by these considerations, the paper is organized as follows. Section~2 introduces the necessary preliminaries on quaternion tensors, the QT-product, and $z$-block circulant matrices.
Section~3 revisits the QT-SVD under the QT-product by combining the complex representation of quaternion matrices with the algebraic properties of $z$-block circulant matrices, and
establishes the corresponding $z$-block circulant matrix decomposition. It is worth noting that we remove an unnecessary requirement in the proof of \cite[Theorem~2.2]{AMLzhang2022}. We further develop a fast QT-SVD algorithm based on the complex representation of quaternion matrices and compare its computational efficiency and numerical accuracy with those of an implementation using quaternion matrix SVD \cite{AMLzhang2022}. Section~4 applies the QT-SVD to color video encryption and reconstruction and evaluates its performance against the quaternion T-SVD \cite{tsvdarxiv} in terms of reconstruction quality, ciphertext statistics, differential sensitivity, stability under input perturbations, and computational time.

%%%%%%%%%%%%%%%%%%%%%%%%%%%%%%%%%%%%%%%%%%%%%%%%%%
%%%%%%%%%%%%%%%%%%%%%%%%%%%%%%%%%%%%%%%%%%%%%%%%%%
\section{Preliminaries and key lemmas}
Let \(\mathbb R\), \(\mathbb C\), and \(\mathbb H\) denote the real field, the complex field, and the quaternion division ring, respectively. The quaternion algebra \(\mathbb{H}\) is a four-dimensional vector space over \(\mathbb{R}\) with basis \(\{1,\mathbf{i},\mathbf{j},\mathbf{k}\}\). Let \(\mathbb{H}^{n}\) denote the set of \(n\)-dimensional quaternion vectors, and let \(Q\in\mathbb{H}^{m\times n}\) denote an \(m\times n\) quaternion matrix. For a matrix $A$, the notation $A^*$ denotes its conjugate transpose.

Throughout the paper, \(\mathcal{A}\in\mathbb{H}^{n_1\times n_2\times n_3}\) denotes a third-order quaternion tensor, which is represented as
$
\mathcal{A}=\mathcal{A}_{\mathbf d}+\mathbf{j}\mathcal{A}_{\mathbf c},
$
where \(\mathcal{A}_{\mathbf d},\,\mathcal{A}_{\mathbf c}\in\mathbb{C}^{n_1\times n_2\times n_3}\). For a third-order tensor \(\mathcal{A}\), we adopt MATLAB-style indexing: \(\mathcal{A}(i,:,:)\), \(\mathcal{A}(:,j,:)\), and \(\mathcal{A}(:,:,k)\) denote the \(i\)-th horizontal, \(j\)-th lateral, and \(k\)-th frontal slices, respectively. For conciseness, the \(k\)-th frontal slice is denoted by \(\mathcal{A}^{(k)}:=\mathcal{A}(:,:,k)\). For a third-order quaternion tensor
$\mathcal A=(a_{ijk})\in\mathbb H^{n_1\times n_2\times n_3}$,
its Frobenius norm is defined by
\[
\|\mathcal A\|_F
=
\left(
\sum_{i=1}^{n_1}
\sum_{j=1}^{n_2}
\sum_{k=1}^{n_3}
|a_{ijk}|^2
\right)^{1/2},
\]
where $|a_{ijk}|$ denotes the modulus of the quaternion entry $a_{ijk}$.

For a quaternion matrix $A\in\mathbb H^{m\times n}$, write
$A=X+Y\mathbf j$, where $X,Y\in\mathbb C^{m\times n}$.
Following \cite{LAAzfzl1997}, the complex
representation of $A$ is defined by
\[
\chi(A)=
\begin{bmatrix}
X & Y\\
-\overline{Y} & \overline{X}
\end{bmatrix}
\in\mathbb C^{2m\times 2n},
\]
where the overline denotes complex conjugation.
The map $\chi$ is injective and satisfies
\[
\chi(AB)=\chi(A)\chi(B),
\qquad
\chi(A^*)=\chi(A)^*.
\]
Moreover, the singular values of $\chi(A)$
are precisely those of $A$, each with twice the multiplicity.

For the representation $A=A_d+\mathbf j A_c$ used above, we have $A=A_d+\overline{A_c}\mathbf j,$
since $\mathbf j z=\overline z\,\mathbf j$ for $z\in\mathbb C$. Thus, $X=A_d$ and $Y=\overline{A_c}$.

The block circulant matrix associated with \(\mathcal{A}\in \mathbb{H}^{n_1 \times n_2 \times n_3}\), denoted by $\mathtt{bcirc}({\mathcal{A})}$, has dimension \(n_1n_3\times n_2n_3\) and is defined as
\begin{align*}
    \mathtt{bcirc}(\mathcal{A}) :=
\begin{bmatrix}
\mathcal A^{(1)} &\mathcal A^{(n_3)} & \mathcal A^{(n_3-1)} & \cdots & \mathcal A^{(2)} \\
\mathcal A^{(2)} & \mathcal A^{(1)} & \mathcal A^{(n_3)} & \cdots & \mathcal A^{(3)} \\
\vdots & \vdots & \vdots & \ddots & \vdots \\
\mathcal A^{(n_3)} & \mathcal A^{(n_3-1)} & \mathcal A^{(n_3-2)} & \cdots & \mathcal A^{(1)}
\end{bmatrix}.
\end{align*}

The unfolding operator maps a third-order tensor
\(\mathcal{A}\in\mathbb{H}^{n_1\times n_2\times n_3}\)
to the first block column of
\(\mathtt{bcirc}(\mathcal{A})\). Its inverse operator reconstructs the original tensor from the unfolded representation. Specifically,
\[
\mathtt{unfold}(\mathcal A)
=
\begin{bmatrix}
\mathcal A^{(1)}\\
\mathcal A^{(2)}\\
\vdots\\
\mathcal A^{(n_3)}
\end{bmatrix},
\qquad
\mathtt{fold}(\mathtt{unfold}(\mathcal A))
=\mathcal A.
\]

Let $F_{n_3}\in\mathbb C^{n_3\times n_3}$ denote the
normalized discrete Fourier transform (DFT) matrix defined by
\[
F_{n_3}(i,j)
=n_3^{-1/2}\omega^{(i-1)(j-1)},
\qquad
\omega=e^{-2\pi\mathbf i/n_3},
\quad i,j=1,\ldots,n_3.
\]
For a quaternion tensor
$\mathcal A\in\mathbb H^{n_1\times n_2\times n_3}$,
its DFT along the third mode is defined by
\cite{AMLzhang2022}
\[
\widehat{\mathcal A}(i,j,:)
=
\sqrt{n_3}F_{n_3}\mathcal A(i,j,:),
\qquad
1\le i\le n_1,\quad 1\le j\le n_2.
\]
In MATLAB, the above transform is implemented by
\(fftq(\mathcal A,[~],3)\),
and its inverse by
\(ifftq(\hat{\mathcal A},[~],3)\).
Equivalently, the DFT and its inverse can be expressed in terms of the unfolding operator as
\begin{align}\label{zhao2.1}
\begin{aligned}
\mathtt{unfold}(\hat{\mathcal A})
&=
\sqrt{n_3}(F_{n_3}\otimes I_{n_1})
\mathtt{unfold}(\mathcal A),\\
\mathtt{unfold}(\mathcal A)
&=
\frac{1}{\sqrt{n_3}}
(F_{n_3}^{*}\otimes I_{n_1})
\mathtt{unfold}(\hat{\mathcal A}).
\end{aligned}
\end{align}

We introduce several basic definitions that will be used in the subsequent sections.
\begin{definition}\label{bcircz}
    Let $\mathcal A\in \mathbb H^{n_1\times r\times n_3}$, $\mathcal{A}_{\mathbf{d}}\in \mathbb C^{n_1\times r\times n_3}$, and $\mathcal{A}_{\mathbf{c}}\in \mathbb C^{n_1\times r\times n_3}$ satisfy that $\mathcal A=\mathcal{A}_{\mathbf{d}} + \mathbf{j} \mathcal{A}_{\mathbf{c}}\in \mathbb {H}^{n_1\times r\times n_3}.$
    Then the z-block circulant matrix of $\mathcal A$ is defined by $\mathtt{bcirc_z}(\mathcal A)$ as follows:
    \begin{align*}%\label{defbcirc_z}
    \mathtt{bcirc_z}(\mathcal A)=\mathtt{bcirc}(\mathcal{A}_{\mathbf{d}})+\mathbf{j}\mathtt{bcirc}(\mathcal{A}_{\mathbf{c}})(P_{n_3}\otimes I_r),
    \end{align*}
    where $P_{n_3}$ is the permutation matrix defined by \(P_{11}=1\) and
\begin{align}\label{permutationmatrix}
P_{ij}=
\begin{cases}
1,& i+j=n_3+2,\quad 2\le i,j\le n_3,\\
0,& \text{otherwise}.
\end{cases}
\end{align}
\end{definition}

From \cite{dczblock}, the $z$-block circulant matrix has the following form:
    \begin{align}\label{bcirczA}
        \mathtt{bcirc_z}(\mathcal A)=
        \begin{bmatrix}
        \mathcal{A}_{\mathbf{d}}^{(1)} + \mathbf{j}\mathcal{A}_{\mathbf{c}}^{(1)} & \mathcal{A}_{\mathbf{d}}^{(n_3)} +  \mathbf{j}\mathcal{A}_{\mathbf{c}}^{(2)} & \cdots & \mathcal{A}_{\mathbf{d}}^{(2)} + \mathbf{j}\mathcal{A}_{\mathbf{c}}^{(n_3)} \\
        \mathcal{A}_{\mathbf{d}}^{(2)} + \mathbf{j}\mathcal{A}_{\mathbf{c}}^{(2)} & \mathcal{A}_{\mathbf{d}}^{(1)} + \mathbf{j}\mathcal{A}_{\mathbf{c}}^{(3)} & \cdots & \mathcal{A}_{\mathbf{d}}^{(3)} + \mathbf{j}\mathcal{A}_{\mathbf{c}}^{(1)}  \\
        \vdots & \vdots & \ddots & \vdots \\
        \mathcal{A}_{\mathbf{d}}^{(n_3)} + \mathbf{j}\mathcal{A}_{\mathbf{c}}^{(n_3)} & \mathcal{A}_{\mathbf{d}}^{(n_3-1)} + \mathbf{j}\mathcal{A}_{\mathbf{c}}^{(1)} & \cdots & \mathcal{A}_{\mathbf{d}}^{(1)} + \mathbf{j}\mathcal{A}_{\mathbf{c}}^{(n_3-1)}
        \end{bmatrix}.
    \end{align}

In what follows, we present the definition of the QT-product.
\begin{definition}\cite{AMLzhang2022}\label{defAB}
Let $\mathcal{A}_{\mathbf{d}},~\mathcal{A}_{\mathbf{c}} \in \mathbb{C}^{n_1 \times r \times n_3}$, 
and consider tensors $\mathcal{A} = \mathcal{A}_{\mathbf{d}} + \mathbf{j}\mathcal{A}_{\mathbf{c}} 
\in \mathbb{H}^{n_1 \times r \times n_3}$ and 
$\mathcal{B} \in \mathbb{H}^{r \times n_2 \times n_3}$. 
Here, the symbol~$\otimes$ denotes the Kronecker product. We define
\begin{align*}
    \mathcal A*_Q\mathcal B\stackrel{.}{=}\mathtt{fold}\Big(\big(\mathtt{bcirc}(\mathcal A_{\mathbf{d}})+\mathbf j\mathtt{bcirc}(\mathcal A_{\mathbf{c}})\cdot (P_{n_3}\otimes I_r)\big)\cdot\mathtt{unfold}(\mathcal B)\Big)\in \mathbb{H}^{n_1 \times n_2 \times n_3}.
    \end{align*}
\end{definition}

We recall the conjugate transpose and the definitions of the identity, unitary, and f-diagonal quaternion tensors. For a complex tensor $\mathcal B\in\mathbb C^{n_1\times n_2\times n_3}$, its conjugate transpose $\mathcal B^*$ is defined by
\[
(\mathcal B^*)^{(1)}=(\mathcal B^{(1)})^*,
\qquad
(\mathcal B^*)^{(k)}
=(\mathcal B^{(n_3-k+2)})^*,
\quad k=2,\ldots,n_3,
\]
where the superscript $*$ on the right denotes the
conjugate transpose of a complex matrix.
\begin{definition}\cite{AMLzhang2022}
The conjugate transpose of a quaternion tensor 
$
\mathcal{A}=\mathcal A_{\mathbf{d}}+\mathbf{j}\mathcal A_{\mathbf{c}}
    \in\mathbb{H}^{n_1\times n_2\times n_3},
$
denoted by $\mathcal{A}^*\in\mathbb{H}^{\,n_2\times n_1\times n_3}$, is defined by
$
\mathtt{unfold}(\mathcal{A}^*)
    = \mathtt{unfold}(\mathcal A_{\mathbf{d}}^*)
      - (P_{n_3}\otimes I_{n_2})\,\mathtt{unfold}(\mathcal A_{\mathbf{c}}^*)\,\mathbf{j}.
$
\end{definition}

\begin{definition}\cite{AMLzhang2022}
The $n \times n \times n_3$ identity quaternion tensor $\mathcal{I}_{nnn_3}$ is the tensor whose first frontal slice is the identity matrix, and whose remaining frontal slices are zero matrices.
\end{definition}

\begin{definition}\cite{AMLzhang2022}\label{defunitary}
A quaternion tensor
$\mathcal U\in\mathbb H^{n\times n\times n_3}$ is called unitary if
$
\mathcal U^{*} *_Q \mathcal U
=
\mathcal U *_Q \mathcal U^{*}
=
\mathcal I_{n\times n\times n_3}.
$
\end{definition}

\begin{definition}\cite{AMLzhang2022,ZMMAMLD}\label{fdiagonal}
The tensor \(\mathcal{P} \in \mathbb{H}^{n_1 \times n_2 \times n_3}\) is called an \emph{f-diagonal} quaternion tensor if each of its frontal slices is a diagonal quaternion matrix.
\end{definition}
    
We first recall the Fourier block diagonalization of complex block circulant matrices. For
$\mathcal A\in\mathbb C^{n_1\times n_2\times n_3}$,
we have \cite{gloub}
\[
(F_{n_3}\otimes I_{n_1})
\mathtt{bcirc}(\mathcal A)
(F_{n_3}^{*}\otimes I_{n_2})
=
\operatorname{diag}
\bigl(
\widehat{\mathcal A}^{(1)},
\ldots,
\widehat{\mathcal A}^{(n_3)}
\bigr)
:=\operatorname{diag}(\widehat{\mathcal A}),
\]
where $\widehat{\mathcal A}^{(k)}$ denotes
the $k$-th frontal slice of $\widehat{\mathcal A}$.

The following lemmas give the corresponding Fourier block diagonalization and algebraic properties of quaternion $z$-block circulant matrices.

\begin{lemma}\cite{dczblock}\label{zdc}
Let $\mathcal{A}_{\mathbf{d}}$ and $\mathcal{A}_{\mathbf{c}}\in \mathbb C^{n_1\times n_2\times n_3}$ satisfy that $\mathcal A=\mathcal{A}_{\mathbf{d}} + \mathbf{j} \mathcal{A}_{\mathbf{c}}\in\mathbb{H}^{{n_1\times n_2\times n_3}}.$ Denote that $\hat{\mathcal A}$ is the DFT of $\mathcal{A}$. Then, 
$
    (F_{n_3}\otimes I_{n_1})\mathtt{bcirc_z}(\mathcal A)(F_{n_3}^*\otimes I_{n_2})=diag(\widehat{\mathcal A}).
$
\end{lemma}

\begin{lemma}\cite{dczblock}\label{newrelation}
    Let $\mathcal A\in \mathbb H^{n_1\times r\times n_3},$ $\mathcal B\in \mathbb H^{r\times n_2\times n_3},$ and $\mathcal C\in \mathbb H^{n_1\times n_2\times n_3}.$  Then
    \begin{align*}
       \mathcal A*_Q\mathcal B=\mathcal C\Longleftrightarrow \mathtt{bcirc_z}(\mathcal A)\mathtt{bcirc_z}(\mathcal B)=\mathtt{bcirc_z}(\mathcal C).
    \end{align*}
    \end{lemma}
    
\begin{lemma}\cite{dczblock}\label{ct}
    Let $\mathcal A\in\mathbb H^{n\times n\times n_3}$ and the conjugate transpose of $\mathcal A$ is defined by $\mathcal A^*$. Then $\mathtt{bcirc_z}(\mathcal A^*)=\mathtt{bcirc_z}(\mathcal A)^*.$
    \end{lemma}

\begin{lemma}\cite{dczblock}\label{unitary}
        A quaternion tensor $\mathcal A\in\mathbb{H}^{n\times n\times n_3}$ is unitary if and only if $\mathtt{bcirc_z}(\mathcal A)$ is unitary.
    \end{lemma}

%%%%%%%%%%%%%%%%%%%%%%%%%%%%%%%%%%%%%%%%%%%%%%%%%%%%%%%%%%%%%%%%%%%%%%%%%%%%%%%%%%%%%%%%%%%%%%%%%%%%
\section{Quaternion Tensor Singular Value Decomposition and Its Fast Implementation}
\subsection{QT-SVD under the QT-Product}
This subsection revisits the quaternion tensor singular value decomposition (QT-SVD) under the QT-product and establishes the corresponding $z$-block circulant matrix decomposition. The proof also eliminates an unnecessary assumption that $\operatorname{diag}(\widehat{\mathcal S})$
is diagonal in \cite[Theorem~2.2]{AMLzhang2022}.
\begin{theorem}[\textbf{QT-SVD}]
    Let $\mathcal{A}\in\mathbb{H}^{n_1\times n_2\times n_3}$. Then there exist unitary quaternion tensors $\mathcal U\in\mathbb{H}^{n_1\times n_1\times n_3}$, $\mathcal V\in\mathbb{H}^{n_2\times n_2\times n_3}$, and an f-diagonal tensor $\mathcal S\in\mathbb{H}^{n_1\times n_2\times n_3}$ such that $\mathcal{A}$ can be factored  as
    \begin{align*}
        \mathcal A=\mathcal U*_Q\mathcal S*_Q\mathcal V^*,
    \end{align*}
    where $\mathcal S$ is the singular-value tensor, referred to as the core tensor in the encryption procedure.
    Meanwhile, 
    \begin{align*}
        \mathtt{bcirc_z}(\mathcal A)&=\mathtt{bcirc_z}(\mathcal U)\mathtt{bcirc_z}(\mathcal S)\mathtt{bcirc_z}(\mathcal V)^*,
        \end{align*} 
        where $\mathtt{bcirc_z}(\mathcal U)$ and $\mathtt{bcirc_z}(\mathcal V)$ are both unitary and every block of $\mathtt{bcirc_z}(\mathcal S)$
is diagonal.
\end{theorem}
    \begin{proof}
    For each $i=1,2,\ldots,n_3$, write the $i$-th frontal slice of
$\widehat{\mathcal A}$ as
$\widehat{\mathcal A}^{(i)}
=X^{(i)}+Y^{(i)}\mathbf j,~
X^{(i)},Y^{(i)}\in\mathbb C^{n_1\times n_2}.$
Its complex representation is
\[
C_i
=
\begin{bmatrix}
X^{(i)} & Y^{(i)}\\[1mm]
-\overline{Y^{(i)}} & \overline{X^{(i)}}
\end{bmatrix}
\in\mathbb C^{2n_1\times 2n_2},
\]
where
\[
\chi(X+Y\mathbf j)
=
\begin{bmatrix}
X & Y\\
-\overline{Y} & \overline{X}
\end{bmatrix}.
\]
Consider a complex singular value decomposition
$C_i=U_{c,i}\Sigma_{c,i}V_{c,i}^*$.
Since $C_i$ is the complex representation of a quaternion matrix, its singular values occur in pairs. Moreover, the singular vectors can be chosen to preserve the quaternion structure, so that
\[
U_{c,i}
=
\begin{bmatrix}
P_i & -\overline{Q_i}\\
Q_i & \overline{P_i}
\end{bmatrix},
\qquad
V_{c,i}
=
\begin{bmatrix}
R_i & -\overline{T_i}\\
T_i & \overline{R_i}
\end{bmatrix},
\]
and
\[
\Sigma_{c,i}
=
\begin{bmatrix}
\widehat S^{(i)} & 0\\
0 & \widehat S^{(i)}
\end{bmatrix},
\]
where $\widehat S^{(i)}\in\mathbb R^{n_1\times n_2}$ is diagonal with
nonnegative diagonal entries.

Define
$\widehat U^{(i)}=P_i-\overline{Q_i}\mathbf j,$ and $\widehat V^{(i)}=R_i-\overline{T_i}\mathbf j.$ Then
\[
U_{c,i}=\chi\bigl(\widehat U^{(i)}\bigr),
\qquad
V_{c,i}=\chi\bigl(\widehat V^{(i)}\bigr),
\qquad
\Sigma_{c,i}=\chi\bigl(\widehat S^{(i)}\bigr),
\]
Since $\chi$ is multiplicative and preserves the conjugate transpose,
\[
\begin{aligned}
\chi\bigl(\widehat{\mathcal A}^{(i)}\bigr)
=
C_i=U_{c,i}\Sigma_{c,i}V_{c,i}^*=
\chi\bigl(\widehat U^{(i)}\bigr)
\chi\bigl(\widehat S^{(i)}\bigr)
\chi\bigl(\widehat V^{(i)}\bigr)^*=
\chi\left(
\widehat U^{(i)}
\widehat S^{(i)}
\bigl(\widehat V^{(i)}\bigr)^*
\right).
\end{aligned}
\]
By the injectivity of $\chi$,
\[
\widehat{\mathcal A}^{(i)}
=
\widehat U^{(i)}
\widehat S^{(i)}
\bigl(\widehat V^{(i)}\bigr)^*,
\qquad
i=1,2,\ldots,n_3.
\]
Furthermore, since $U_{c,i}$ and $V_{c,i}$ are unitary,
\[
\chi\left(
\bigl(\widehat U^{(i)}\bigr)^{*}\widehat U^{(i)}
\right)
=
U_{c,i}^{*}U_{c,i}
=
I
=
\chi(I),
\]
and similarly,
$
\chi\left(\bigl(\widehat V^{(i)}\bigr)^{*}\widehat V^{(i)}
\right)=I.$
The injectivity of $\chi$ therefore gives
\[
\bigl(\widehat U^{(i)}\bigr)^{*}\widehat U^{(i)}=I,
\qquad
\bigl(\widehat V^{(i)}\bigr)^{*}\widehat V^{(i)}=I.
\]
Thus $\widehat U^{(i)}$ and $\widehat V^{(i)}$ are unitary quaternion
matrices, and the stated quaternion SVD follows for each frontal slice

Consequently,
\[
\operatorname{diag}(\widehat{\mathcal A})
=
\operatorname{diag}(\widehat{\mathcal U})
\operatorname{diag}(\widehat{\mathcal S})
\operatorname{diag}(\widehat{\mathcal V})^*.
\]
Since every frontal slice $\widehat S^{(i)}$ is diagonal with
nonnegative real diagonal entries, $\mathcal S$ is an $f$-diagonal
quaternion tensor. Moreover, by Lemma~\ref{zdc},
\[
\operatorname{bcirc}_z(\mathcal U)
=
(F_{n_3}^*\otimes I_{n_1})
\operatorname{diag}(\widehat{\mathcal U})
(F_{n_3}\otimes I_{n_1}),
\]
and hence $\operatorname{bcirc}_z(\mathcal U)$ is unitary. It follows from Lemma~\ref{unitary} that $\mathcal U$ is unitary. The same argument shows that $\mathcal V$
is unitary. Applying Lemma~\ref{zdc}, we obtain
\[
\operatorname{bcirc}_z(\mathcal A)
=
\operatorname{bcirc}_z(\mathcal U)
\operatorname{bcirc}_z(\mathcal S)
\operatorname{bcirc}_z(\mathcal V)^*.
\]
By Lemma~\ref{ct},
$
\operatorname{bcirc}_z(\mathcal V)^*
=
\operatorname{bcirc}_z(\mathcal V^*),
$ and therefore
\[
\operatorname{bcirc}_z(\mathcal A)
=
\operatorname{bcirc}_z(\mathcal U)
\operatorname{bcirc}_z(\mathcal S)
\operatorname{bcirc}_z(\mathcal V^*).
\]
Finally, Lemma~\ref{newrelation} yields
\[
\mathcal A
=
\mathcal U *_Q \mathcal S *_Q \mathcal V^*.
\]
On the other hand, observing the structure of \eqref{bcirczA}, it is obvious that every block of $\mathtt{bcirc_z}(\mathcal S)$
is diagonal.
\end{proof}

%%%%%%%%%%%%%%%%%%%%%%%%%%%%%%%%%%%%%%%%%%%%%%%%%%%%%%%%%%%%%%%%%%%%%%%%%%%%%%%%%%%%%%%%%%%%    
\begin{remark}
    It is worth noting that $\mathcal S$ is f-diagonal equivalent to $\hat{\mathcal S}$ being f-diagonal from \eqref{zhao2.1}. However, the proof in \cite[Theorem 2.2]{AMLzhang2022} mentions that $diag(\hat{\mathcal S})$ is diagonal. Actually, this extra limitation is not required. 
\end{remark}

In what follows, we provide a numerical example to illustrate the truth. 
% \begin{algorithm}[H] \caption{QT-SVD \cite{LiuAMC2026}}\label{QTSVD}
% \begin{algorithmic}[1]
% \smallskip
% \Input Third-order quaternion tensor $\mathcal{A} \in \mathbb{H}^{n_1 \times n_2 \times n_3}.$
% \Output  Third-order unitary quaternion tensors $\mathcal{U}\in \mathbb{H}^{n_1 \times n_1 \times n_3}$ and $\mathcal {V} \in \mathbb{H}^{n_2 \times n_2 \times n_3}$, third-order f-diagonal quaternion tensor $\mathcal{S} \in \mathbb{H}^{n_1 \times n_2 \times n_3}$.
% \State $\hat{\mathcal A}=fftq(\mathcal A,[\ ],3)$
% \For {$i=1,2,...,n_3$ }

% $[\hat{\mathcal{U}}(:,:,i),\hat{\mathcal{S}}(:,:,i),\hat{\mathcal{V}}(:,:,i)]=svd(\hat{\mathcal A}(:,:,i))$
% \EndFor
% \State 
% $\mathcal{U}=ifftq(\hat{\mathcal{U}},[\ ],3),\quad \mathcal{S}=ifftq(\hat{\mathcal{S}},[\ ],3), \quad \mathcal{V}=ifftq(\hat{\mathcal{V}},[\ ],3)$

% \end{algorithmic}
% \end{algorithm}

\begin{example}
Let $\mathcal A=\mathcal A_{\mathbf{d}}+\mathbf{j}\mathcal A_{\mathbf{c}}\in\mathbb{H}^{3\times 2\times 3}$ whose frontal slices are given below.
{\small
     \[\mathcal{A}_{\mathbf{d}}(:,:,1) = \begin{bmatrix} \ -0.4102-0.201\mathbf i & -0.6358-0.4573 \mathbf i\ \\ \ -0.5639+0.1595\mathbf i & 0.4203+0.0335\mathbf i\ \\ \ -0.6145-0.2805\mathbf i & 0.09-0.6297\mathbf i\ \end{bmatrix},~
        \mathcal{A}_{\mathbf{c}}(:,:,1) = \begin{bmatrix} \ -0.3436-0.8205\mathbf i & 0.169-0.5984\mathbf i\ \\ \ 0.1612-0.7941\mathbf i & 0.4122+0.8077\mathbf i\ \\ \ -0.4092+0.6133\mathbf i & -0.1043+0.7645\mathbf i\ \end{bmatrix},
   \]\small}{\small
     \[\mathcal{A}_{\mathbf{d}}(:,:,2) = \begin{bmatrix} \ -0.1818+0.4305\mathbf i & 0.3797-0.054\mathbf i\ \\ \ 0.1108+0.0062\mathbf i & -0.7078+0.1158\mathbf i\ \\ \ 0.4451+0.1169\mathbf i & 0.3098-0.1739\mathbf i\ \end{bmatrix},~\mathcal{A}_{\mathbf{c}}(:,:,2) = \begin{bmatrix} \ -0.7455-0.4753\mathbf i & -0.5206-0.7628\mathbf i\ \\ \ -0.7029+0.7026\mathbf i & -0.6927-0.0756\mathbf i\ \\ \ 0.8874+0.029\mathbf i & -0.3586+0.8632\mathbf i\ \end{bmatrix},
  \]\small}{\small
     \[~
     \mathcal{A}_{\mathbf{d}}(:,:,3) = \begin{bmatrix} \ -0.4924-0.6509\mathbf i & 0.9176-0.3561\mathbf i\ \\ \ 0.8159+0.3591\mathbf i & 0.4122+0.0498\mathbf i\ \\ \ 0.3714+0.6064\mathbf i & -0.4633+0.7219\mathbf i\ \end{bmatrix},~
     \mathcal{A}_{\mathbf{c}}(:,:,3) = \begin{bmatrix} \ -0.3479-0.4613\mathbf i & -0.12+0.1296\mathbf i\ \\ \ 0.334+0.3062\mathbf i & 0.4724-0.7775\mathbf i\ \\ \ 0.599+0.3682\mathbf i & 0.3624+0.3646\mathbf i\ \end{bmatrix}.\]
   In accordance with Algorithm 2 in \cite{LiuAMC2026}, applying the $fftq(\mathcal{A},[\ ],3)$ operation to \(\mathcal{A}\) along the third mode yields the transformed tensor \(\hat{\mathcal{A}}=\hat{\mathcal A}_{\mathbf{d}}+\mathbf{j}\hat{\mathcal A}_{\mathbf{c}},\) where
     {\small
     \[\hat{\mathcal{A}}_{\mathbf{d}}(:,:,1) = \begin{bmatrix} \ -1.0844-0.4214\mathbf i & 0.6615-0.8674\mathbf i\ \\ \ 0.3629+0.5248\mathbf i & 0.1247+0.1991\mathbf i\ \\ \ 0.2019+0.4428\mathbf i & -0.0635-0.0818\mathbf i\ \end{bmatrix},~
       \hat{\mathcal{A}}_{\mathbf{c}}(:,:,1) = \begin{bmatrix} \ -1.437-1.7572\mathbf i & -0.4715-1.2316\mathbf i\ \\ \ -0.2076+0.2147\mathbf i & 0.1918-0.0454\mathbf i\ \\ \ 1.0771+1.0106\mathbf i & -0.1006+1.9923\mathbf i\ \end{bmatrix},
   \]\small}{\small
     \[\hat{\mathcal{A}}_{\mathbf{d}}(:,:,2) = \begin{bmatrix} \ 0.8634-0.3598\mathbf i & -1.0228+0.2136\mathbf i\ \\ \ -1.3329+0.5874\mathbf i & 0.6253+0.9207\mathbf i\ \\ \ -1.4466-0.706\mathbf i & -0.609-1.5733\mathbf i\ \end{bmatrix},~
     \hat{\mathcal{A}}_{\mathbf{c}}(:,:,2) = \begin{bmatrix} \ 0.2152-0.6965\mathbf i & 1.2622-0.6287\mathbf i\ \\ \ 0.0023-2.1965\mathbf i & -0.0855+0.2252\mathbf i\ \\ \ -0.8586+0.6645\mathbf i & -0.538-0.4738\mathbf i\ \end{bmatrix},\]\small}{\small
     \[\hat{\mathcal{A}}_{\mathbf{d}}(:,:,3) = \begin{bmatrix} \ -1.0097+0.1783\mathbf i & -1.5461-0.7181\mathbf i\ \\ \ -0.7216-0.6338\mathbf i & 0.5109-1.0193\mathbf i\ \\ \ -0.5988-0.5784\mathbf i & 0.9425-0.2342\mathbf i\ \end{bmatrix},~
     \hat{\mathcal{A}}_{\mathbf{c}}(:,:,3) = \begin{bmatrix} \ 0.1909-0.008\mathbf i & -0.2836+0.0651\mathbf i\ \\ \ 0.689-0.4005\mathbf i & 1.1302+2.2432\mathbf i\ \\ \ -1.4462+0.1649\mathbf i & 0.3256+0.7749\mathbf i\ \end{bmatrix}.\]\small}
     
    Then performing the SVD on each frontal slice of $\hat{\mathcal{A}}$ 
     to obtain $\hat{\mathcal{U}}$, $\hat{\mathcal{S}}$, and $\hat{\mathcal{V}}$. The frontal slices of $\hat{\mathcal{U}}
= \hat{\mathcal{U}}_{\mathbf{d}}
+\mathbf{j}\hat{\mathcal{U}}_{\mathbf{c}}$ are listed below.
    \[\hat{\mathcal{U}}_{\mathbf{d}}(:,:,1) = \begin{bmatrix} \ -0.316-0.2197\mathbf i & 0.194-0.3172\mathbf i & -0.022+0.2968\mathbf i\ \\ \ 0.1119+0.1231\mathbf i & -0.0399-0.2149\mathbf i & -0.3006+0.2423\mathbf i\ \\ \ 0.1197-0.0027\mathbf i & 0.2103-0.613\mathbf i & 0.0389+0.1057\mathbf i\ \end{bmatrix},\]
     \[\hat{\mathcal{U}}_{\mathbf{d}}(:,:,2) = \begin{bmatrix} \ 0.0741-0.1092\mathbf i & -0.553-0.0001\mathbf i & -0.2588+0.1853\mathbf i\ \\ \ -0.2397+0.1718\mathbf i & 0.4848-0.0189\mathbf i & -0.3055+0.1774\mathbf i\ \\ \ -0.408-0.2385\mathbf i & 0.0908-0.0717\mathbf i & -0.3391+0.2349\mathbf i\ \end{bmatrix},\]
     \[\hat{\mathcal{U}}_{\mathbf{d}}(:,:,3) = \begin{bmatrix} \ -0.3258+0.2581\mathbf i & 0.2858+0.2151\mathbf i & 0.1633+0.026\mathbf i\ \\ \ -0.338-0.583\mathbf i & 0.0651-0.3218\mathbf i & -0.3582-0.3612\mathbf i\ \\ \ -0.1246-0.3518\mathbf i & 0.2604-0.055\mathbf i & -0.0401+0.3709\mathbf i\ \end{bmatrix},\]
     \[\hat{\mathcal{U}}_{\mathbf{c}}(:,:,1) = \begin{bmatrix} \ -0.4497-0.4981\mathbf i & 0.1318+0.3711\mathbf i & -0.1335+0.0429\mathbf i\ \\ \ -0.0402+0.019\mathbf i & 0.1195-0.215\mathbf i & -0.1106-0.8371\mathbf i\ \\ \ 0.4208+0.4248\mathbf i & 0.2405+0.3472\mathbf i & -0.0494+0.1207\mathbf i\ \end{bmatrix},\]
     \[\hat{\mathcal{U}}_{\mathbf{c}}(:,:,2) = \begin{bmatrix} \ -0.0157-0.3168\mathbf i & -0.2385-0.3107\mathbf i & -0.0244+0.5664\mathbf i\ \\ \ 0.0912-0.5875\mathbf i & 0.2666+0.2318\mathbf i & 0.0911+0.2573\mathbf i\ \\ \ -0.3572+0.3011\mathbf i & -0.2851+0.2932\mathbf i & -0.4536+0.0441\mathbf i\ \end{bmatrix},\]
     \[\hat{\mathcal{U}}_{\mathbf{c}}(:,:,3) = \begin{bmatrix} \ 0.1314-0.2113\mathbf i & 0.0384-0.2741\mathbf i & 0.3643+0.633\mathbf i\ \\ \ 0.2915-0.2155\mathbf i & -0.1034-0.0009\mathbf i & -0.0846+0.1732\mathbf i\ \\ \ -0.1881+0.071\mathbf i & 0.7782-0.0232\mathbf i & 0.0456-0.0452\mathbf i\ \end{bmatrix}.\]
Also, each frontal slice of $\hat{\mathcal{V}}
= \hat{\mathcal{V}}_{\mathbf{d}}+ \mathbf{j}\hat{\mathcal{V}}_{\mathbf{c}}$ is listed as follows.
     \[\hat{\mathcal{V}}_{\mathbf{d}}(:,:,1) = \begin{bmatrix} \ 0.7662 & 0.4201+0.4152\mathbf i\ \\ \ -0.6426 & 0.5009+0.4951\mathbf i\ \end{bmatrix},\hat{\mathcal{V}}_{\mathbf{c}}(:,:,1) = \begin{bmatrix} \ 0 & 0.1618+0.1945\mathbf i\ \\ \ 0 & 0.1929+0.232\mathbf i\ \end{bmatrix},\]
     \[\hat{\mathcal{V}}_{\mathbf{d}}(:,:,2) = \begin{bmatrix} \ 0.8947 & 0.1687+0.2127\mathbf i\ \\ \ -0.4466 & 0.3379+0.4262\mathbf i\ \end{bmatrix},\hat{\mathcal{V}}_{\mathbf{c}}(:,:,2) = \begin{bmatrix} \ 0 & -0.3418-0.0943\mathbf i\ \\ \ 0 & -0.6849-0.1889\mathbf i\ \end{bmatrix},\]
     \[\hat{\mathcal{V}}_{\mathbf{d}}(:,:,3) = \begin{bmatrix} \ 0.4792 & 0.1269+0.5947\mathbf i\ \\ \ -0.8777 & 0.0693+0.3247\mathbf i\ \end{bmatrix},\hat{\mathcal{V}}_{\mathbf{c}}(:,:,3) = \begin{bmatrix} \ 0 & 0.4678-0.4264\mathbf i\ \\ \ 0 & 0.2554-0.2328\mathbf i\ \end{bmatrix}.\]
It is noteworthy that
      \[diag(\hat{\mathcal{S}})(:,:) = \begin{bmatrix} \ 3.8889 & 0 & 0 & 0 & 0 & 0\ \\ \ 0 & 1.1447 & 0 & 0 & 0 & 0\ \\ \ 0 & 0 & 0 & 0 & 0 & 0\ \\ \ 0 & 0 & 3.6848 & 0 & 0 & 0\ \\ \ 0 & 0 & 0 & 2.5063 & 0 & 0\ \\ \ 0 & 0 & 0 & 0 & 0 & 0\ \\ \ 0 & 0 & 0 & 0 & 3.8902 & 0\ \\ \ 0 & 0 & 0 & 0 & 0 & 1.604\ \\ \ 0 & 0 & 0 & 0 & 0 & 0\ \end{bmatrix},\] which is not a diagonal matrix.
      
In what follows, we obtain two unitary tensors $\mathcal{U}$, $\mathcal{V}$, and an f-diagonal tensor $\mathcal{S}$. Specifically, the frontal slices of 
$\mathcal{U}=\mathcal{U}_{\mathbf{d}} + \mathbf{j}\mathcal{U}_{\mathbf{c}}$ are represented as

\[\mathcal{U}_{\mathbf{d}}(:,:,1) = \begin{bmatrix} \ -0.1892-0.0236\mathbf i & -0.0244-0.034\mathbf i & -0.0392+0.1694\mathbf i\ \\ \ -0.1552-0.0961\mathbf i & 0.17-0.1852\mathbf i & -0.3215+0.0195\mathbf i\ \\ \ -0.1376-0.1977\mathbf i & 0.1872-0.2466\mathbf i & -0.1134+0.2372\mathbf i\ \end{bmatrix},\]
     \[\mathcal{U}_{\mathbf{d}}(:,:,2) = \begin{bmatrix} \ 0.0426+0.0174\mathbf i & 0.1713-0.3837\mathbf i & -0.0374-0.0582\mathbf i\ \\ \ -0.0843+0.1379\mathbf i & -0.1924+0.1063\mathbf i & -0.1451+0.1266\mathbf i\ \\ \ 0.0959+0.0157\mathbf i & 0.0164-0.2322\mathbf i & 0.1154-0.152\mathbf i\ \end{bmatrix},\]
     \[\mathcal{U}_{\mathbf{d}}(:,:,3) = \begin{bmatrix} \ -0.1694-0.2135\mathbf i & 0.0471+0.1006\mathbf i & 0.0546+0.1856\mathbf i\ \\ \ 0.3515+0.0812\mathbf i & -0.0175-0.136\mathbf i & 0.1659+0.0962\mathbf i\ \\ \ 0.1614+0.1793\mathbf i & 0.0068-0.1343\mathbf i & 0.0369+0.0206\mathbf i\ \end{bmatrix},\]
     and
\[\mathcal{U}_{\mathbf{c}}(:,:,1) = \begin{bmatrix} \ -0.1113-0.3421\mathbf i & -0.0227-0.0713\mathbf i & 0.0688+0.4141\mathbf i\ \\ \ 0.1142-0.2614\mathbf i & 0.0942+0.0053\mathbf i & -0.0347-0.1355\mathbf i\ \\ \ -0.0415+0.2656\mathbf i & 0.2445+0.2057\mathbf i & -0.1525+0.0399\mathbf i\ \end{bmatrix},\]
     \[\mathcal{U}_{\mathbf{c}}(:,:,2) = \begin{bmatrix} \ -0.1996-0.0355\mathbf i & 0.0667+0.3011\mathbf i & -0.1204-0.0734\mathbf i\ \\ \ -0.1846+0.198\mathbf i & 0.0798-0.217\mathbf i & -0.0137-0.4015\mathbf i\ \\ \ 0.2976+0.1284\mathbf i & 0.0893+0.3777\mathbf i & 0.0773+0.1845\mathbf i\ \end{bmatrix},\]
     \[\mathcal{U}_{\mathbf{c}}(:,:,3) = \begin{bmatrix} \ -0.1388-0.1205\mathbf i & 0.0879+0.1413\mathbf i & -0.082-0.2978\mathbf i\ \\ \ 0.0302+0.0823\mathbf i & -0.0545-0.0034\mathbf i & -0.0622-0.3001\mathbf i\ \\ \ 0.1647+0.0308\mathbf i & -0.0934-0.2362\mathbf i & 0.0257-0.1037\mathbf i\ \end{bmatrix}.\]
  Then the frontal slices of 
$\mathcal{V}=\mathcal{V}_{\mathbf{d}} + \mathbf{j}\mathcal{V}_{\mathbf{c}}$ are given by 
     \[\mathcal{V}_{\mathbf{d}}(:,:,1) = \begin{bmatrix} \ 0.7134 & -0.6556\ \\ \ 0.2385-0.4075\mathbf i & 0.3027-0.4153\mathbf i\ \end{bmatrix},~
     \mathcal{V}_{\mathbf{c}}(:,:,1) = \begin{bmatrix} \ 0 & 0\ \\ \ -0.0959+0.1087\mathbf i & 0.0788+0.0632\mathbf i\ \end{bmatrix},
    \]
     \[ \mathcal{V}_{\mathbf{d}}(:,:,2) = \begin{bmatrix} \ 0.0264+0.12\mathbf i & 0.0065+0.1245\mathbf i\ \\ \ -0.0195+0.0082\mathbf i & 0.1284+0.0377\mathbf i\ \end{bmatrix},~\mathcal{V}_{\mathbf{c}}(:,:,2) = \begin{bmatrix} \ 0 & 0\ \\ \ 0.0629+0.0821\mathbf i & -0.1232+0.1238\mathbf i\ \end{bmatrix},
\]    
\[\mathcal{V}_{\mathbf{d}}(:,:,3) = \begin{bmatrix} \ 0.0264-0.12\mathbf i & 0.0065-0.1245\mathbf i\ \\ \ 0.201-0.0159\mathbf i & 0.0698-0.1174\mathbf i\ \end{bmatrix},~
     \mathcal{V}_{\mathbf{c}}(:,:,3) = \begin{bmatrix} \ 0 & 0\ \\ \ -0.1288-0.3853\mathbf i & -0.1486-0.419\mathbf i\ \end{bmatrix}.\]
Note that the frontal slices of $\mathcal{S}$ are diagonal matrices, as shown below.
\[\mathcal{S}(:,:,1) = \begin{bmatrix} \ 3.8213 & 0\ \\ \ 0 & 1.7517\ \\ \ 0 & 0\ \end{bmatrix}, \quad \mathcal{S}(:,:,2) = \begin{bmatrix} \ 0.0338 & 0\ \\ \ 0 & -0.3035\ \\ \ 0 & 0\ \end{bmatrix}, \quad \mathcal{S}(:,:,3) = \begin{bmatrix} \ 0.0338 & 0\ \\ \ 0 & -0.3035\ \\ \ 0 & 0\ \end{bmatrix}.\]
}

Above all, we observe that $\operatorname{diag}(\hat{\mathcal{S}})$ is not a diagonal matrix. Actually, $diag(\hat{\mathcal S})$ is also not required to be a diagonal matrix. In general, requiring $\operatorname{diag}(\hat{\mathcal S})$ to be diagonal \emph{and} each of its frontal slices to be diagonal leaves only one possibility: the first slice of $\operatorname{diag}(\hat{\mathcal S})$ is diagonal, while all remaining slices are zero. Such a requirement is evidently too restrictive.
\end{example}

% BEGIN ADDED ALGORITHMS
\subsection{Algorithms for QT-SVD}
The preceding theorem gives rise to two implementations of the QT-SVD.
Algorithm~\ref{alg:qt-explicit} is based on the explicit $z$-block circulant representation
and its block diagonalization, whereas Algorithm~\ref{alg:qt-fast} computes the decomposition directly from the frontal slices in the Fourier domain. Throughout, $\mathtt{fftq}$ and $\mathtt{ifftq}$ are used to implement the mode-$3$ DFT and inverse DFT, respectively, following the Fourier-transform convention in \cite{AMLzhang2022}. We denote by $\mathtt{qsvd}_{\mathbb C}$ the fast universal QSVD algorithm in \cite[Algorithm~1]{DCTAML2026}, which computes the quaternion matrix SVD via its complex representation and constructs
structure-preserving singular vectors. We use the complex-pair representation $\mathcal A=\mathcal A_{\mathbf d}+\mathbf j\mathcal A_{\mathbf c}.$
Equivalently, with $\mathbf j$ written on the right,
$\mathcal A=\mathcal A_{\mathbf d}
+\overline{\mathcal A_{\mathbf c}}\mathbf j.$
%%%%%%%%%%%%%%%%%%%%%%%%%%%%%%%%%%%%%
\begin{algorithm}[H]
\caption{QT-SVD using the explicit $z$-block circulant matrix}
\label{alg:qt-explicit}
\small
\begin{algorithmic}[1]
\Input $\mathcal A\in\mathbb H^{n_1\times n_2\times m}$.
\Output $\mathcal A=\mathcal U*_Q\mathcal S*_Q\mathcal V^*$,
with QT-unitary $\mathcal U,\mathcal V$ and an f-diagonal
tensor $\mathcal S$.

\State Compute
\[
D=(F_m\otimes I_{n_1})
\mathtt{bcirc}_z(\mathcal A)
(F_m^*\otimes I_{n_2})
=\operatorname{diag}(D_1,\ldots,D_m).
\]

\For{$t=1,\ldots,m$}
    \State $(U_t,\Sigma_t,V_t)
    \gets\operatorname{qsvd}_{\mathbb C}(D_t)$.
\EndFor

\State Set
$U_F=\operatorname{diag}(U_1,\ldots,U_m)$,
$S_F=\operatorname{diag}(\Sigma_1,\ldots,\Sigma_m)$, and
$V_F=\operatorname{diag}(V_1,\ldots,V_m)$.

\State Compute
\begin{align*}
\mathtt{bcirc}_z(\mathcal U)
&=(F_m^*\otimes I_{n_1})U_F(F_m\otimes I_{n_1}),\\
\mathtt{bcirc}_z(\mathcal S)
&=(F_m^*\otimes I_{n_1})S_F(F_m\otimes I_{n_2}),\\
\mathtt{bcirc}_z(\mathcal V)
&=(F_m^*\otimes I_{n_2})V_F(F_m\otimes I_{n_2}).
\end{align*}

\State Recover $\mathcal U,\mathcal S,\mathcal V$
by applying $\mathtt{fold}$ to the first block column
of each corresponding $z$-block circulant matrix.

\end{algorithmic}
\end{algorithm}
%%%%%%%%%%%%%%%%%%%%%%%%%%%%%%%%%%%%%%%%%%%%%%%%%%%%%%%%%%%%%%%%%%%%%%%%%%%%%%
Algorithm~\ref{alg:qt-explicit} implements the QT-SVD through the explicit
$z$-block circulant representation. By avoiding the explicit construction
of this large matrix and performing the decomposition directly on the
Fourier-domain frontal slices, the computational cost is reduced, leading
to the faster FFT-based implementation in Algorithm~\ref{alg:qt-fast}.
\begin{algorithm}[H]
\caption{Fast QT-SVD using $qsvd_\mathbb C$}
\label{alg:qt-fast}
\small
\begin{algorithmic}[1]
\Input $\mathcal A\in\mathbb H^{n_1\times n_2\times m}$.
\Output $\mathcal A=\mathcal U*_Q\mathcal S*_Q\mathcal V^*$,
with QT-unitary $\mathcal U,\mathcal V$ and an f-diagonal
tensor $\mathcal S$.

\State $\widehat{\mathcal A}
\gets\mathtt{fftq}(\mathcal A,[],3)$.

\For{$t=1,\ldots,m$}
    \State $(\widehat{\mathcal U}^{(t)},
             \widehat{\mathcal S}^{(t)},
             \widehat{\mathcal V}^{(t)})
    \gets\operatorname{qsvd}_{\mathbb C}
    (\widehat{\mathcal A}^{(t)})$.
\EndFor

\State $\mathcal U
\gets\mathtt{ifftq}(\widehat{\mathcal U},[],3)$,
$\mathcal S
\gets\mathtt{ifftq}(\widehat{\mathcal S},[],3)$,
$\mathcal V
\gets\mathtt{ifftq}(\widehat{\mathcal V},[],3)$.
\end{algorithmic}
\end{algorithm}
%%%%%%%%%%%%%%%%%%%%%%%%%%%%%%%%%%%%%%%%%%%%%%%%%%%

%%%%%%%%%%%%%%%%%%%%%%%%%%%%%%%%%%%%%%%%%%%%%%%%%%%%%%%%%%%%%%%%%%%%%%%%%%%%%%%
% END ADDED ALGORITHMS
\subsection{Comparison of Two QT-SVD Implementations}
\label{sec:qt-implementation-comparison}
We compare two QT-SVD implementations based on  quaternion SVD \cite[Theorem~2.3]{AMLzhang2022} and the complex representation used in Algorithm~\ref{alg:qt-fast}, respectively.
Both implementations use the same quaternion FFT,
inverse transform, and QT-product, they differ only in the matrix SVD computation
at each frequency.The direct implementation calls QTFM's quaternion  \texttt{svd}, whereas the complex implementation forms the $2n\times2n$ complex representation
of each $n\times n$ quaternion matrix and recovers quaternion matrices from its complex SVD.

To ensure a fair, reproducible, and statistically stable comparison of the two implementations, we adopt the following experimental protocol. For each size in Table~\ref{tab:qt-implementation-comparison}, we generate
three quaternion tensors whose four real components have independent
standard normal entries, using MATLAB seeds 20260918--20260920.
The basis is $\{\mathbf i,\mathbf j,\mathbf k\}$, and no truncation is
applied. All tested Fourier slices have positive, distinct quaternion
singular values. For each input, both
implementations are warmed up once and timed five times, with their
execution order alternated. The table reports the mean of the three
per-input median times. Timing includes the FFT, all matrix SVDs, and
the inverse transforms; error evaluation and file saving are excluded.

For this comparison, we use the relative reconstruction and unitarity errors
\begin{align*}
r_{\rm rec}
&=\frac{\|\mathcal A-\mathcal U*_Q\mathcal S*_Q\mathcal V^*\|_F}
        {\|\mathcal A\|_F},\\
r_{\rm u}
&=\max_{\mathcal W\in\{\mathcal U,\mathcal V\}}
\frac{\max\{\|\mathcal W^**_Q\mathcal W-\mathcal I\|_F,
             \|\mathcal W*_Q\mathcal W^*-\mathcal I\|_F\}}
     {\|\mathcal I\|_F}.
\end{align*}
The tabulated errors are the maxima over the three inputs of each size.

\begin{table}[htbp]
\centering
\caption{Direct quaternion and complex SVD implementations of QT-SVD.
Errors are relative; times are in seconds.}
\label{tab:qt-implementation-comparison}
\small
\begin{tabular}{llrrr}
\toprule
Tensor size & Matrix SVD & $r_{\rm rec}$ & $r_{\rm u}$ & Time (s)\\
\midrule
\multirow{2}{*}{$2\times 2\times 4$} & Quaternion & $1.046\times10^{-15}$ & $8.838\times10^{-16}$ & 0.011027 \\
 & Complex & $6.235\times10^{-16}$ & $6.625\times10^{-16}$ & 0.002508 \\
\addlinespace[2pt]
\multirow{2}{*}{$64\times 64\times 16$} & Quaternion & $3.267\times10^{-15}$ & $2.745\times10^{-15}$ & 0.977710 \\
 & Complex & $2.661\times10^{-14}$ & $2.183\times10^{-14}$ & 0.063538 \\
\addlinespace[2pt]
\multirow{2}{*}{$128\times 128\times 32$} & Quaternion & $4.714\times10^{-15}$ & $3.866\times10^{-15}$ & 6.398060 \\
 & Complex & $4.215\times10^{-14}$ & $3.365\times10^{-14}$ & 0.976450 \\
\addlinespace[2pt]
\multirow{2}{*}{$256\times 256\times 64$} & Quaternion & $7.735\times10^{-15}$ & $6.065\times10^{-15}$ & 67.537641 \\
 & Complex & $6.697\times10^{-14}$ & $5.239\times10^{-14}$ & 13.184397 \\
\bottomrule
\end{tabular}
\end{table}

Both implementations reconstruct the input and return numerically unitary
tensors. The largest relative reconstruction errors are $7.735\times10^{-15}$
and $6.697\times10^{-14}$ for the direct quaternion and complex implementations,
respectively. The relative difference between their Fourier-domain
singular values is at most $3.026\times10^{-15}$, and the relative difference
between their reconstructed tensors is at most $6.734\times10^{-14}$.
The complex implementation has somewhat larger errors for the three
larger sizes, while retaining high numerical accuracy. It is 4.40--15.39 times
faster in these tests. This comparison demonstrates an implementation
speed advantage for the tested inputs; both implementations compute
the same QT-product decomposition.
% END QT-SVD IMPLEMENTATION COMPARISON

\section{Applications to Color Video Processing}
We apply the QT-SVD to color video processing and compare it with the quaternion T-SVD based on the T-product~\cite{tsvdarxiv}. Both methods are tested on the same five DAVIS sequences~\cite{DAVIS2017,DAVIS2016} at $256\times256$ resolution under identical preprocessing, masking, permutation, and recovery procedures. We compare decomposition accuracy, reconstruction quality, ciphertext statistics and differential sensitivity, stability under pre-encryption perturbations, and decomposition time. The processing scheme is described first, followed by the experimental results.

\subsection{Video Encryption and Decryption Scheme}
\label{sec:scheme}
Let
$
\mathcal A
=
R\mathbf i+G\mathbf j+B\mathbf k
\in\mathbb H^{n\times n\times p}
$
represent a color video, where the RGB components are normalized to
$[0,1]$. Under the QT-product, we compute
\[
\mathcal A
=
\mathcal U\qprod_Q\mathcal S\qprod_Q\mathcal V^{*},
\]
where $\mathcal S$ is the core tensor used in the subsequent encryption
step. For comparison, the quaternion T-SVD is computed under the
T-product, with all tensor operations defined consistently with the
corresponding product. In both cases, the complete decomposition is
retained, with no rank truncation or compression. 
% In MATLAB, the QT-SVD
% branch is implemented using \texttt{qtsvd}, \texttt{qt\_product}, and
% \texttt{conjtran\_ten\_qtproduct}, whereas the T-SVD branch uses
% \texttt{tsvd}, \texttt{quaternion\_t\_product}, and
% \texttt{conjtran\_ten\_tproduct}.
The processing steps are
\[
\text{video}
\longrightarrow
\text{decomposition}
\longrightarrow
\text{core component scaling}
\longrightarrow
\text{XOR masking}
\longrightarrow
\text{Arnold permutation}.
\]
During recovery, the Arnold permutation is inverted first, followed by
XOR decoding and inverse scaling of the core components. The recovered
core is then used to reconstruct the video with the retained factors
under the QT-product and T-product, respectively. The visualizations shown below are obtained from the three imaginary components of the encrypted core tensor.

\subsubsection{Component Normalization and Logistic Keystream Masking}

For the \(l\)-th frontal slice and the \(c\)-th stored quaternion component, 
\(c\in\{1,2,3\}\), let \(P_{l,c}\) denote the corresponding 
real-valued core slice. It is normalized and quantized into an 
unsigned 8-bit representation by
\begin{equation}
\overline P_{l,c}=\operatorname{uint8}\!\left(
255\frac{P_{l,c}-m_{l,c}}{M_{l,c}-m_{l,c}}
\right),
\label{normalization}
\end{equation}
where \(m_{l,c}\) and \(M_{l,c}\) are the minimum and maximum values of \(P_{l,c}\), respectively. If \(M_{l,c}=m_{l,c}\), a dedicated constant-slice treatment is adopted to prevent division by zero. The corresponding extrema are preserved for subsequent inverse normalization. Since the conversion to \texttt{uint8} quantizes the scaled values to integers in \([0,255]\), inverse normalization may introduce a small recovery error.

The Logistic sequence \cite{Logistic} is generated independently for each frame and component by
\begin{equation*}
y_{t+1}=\mu y_t(1-y_t),\qquad \mu=3.97,
\end{equation*}
where the initial state is defined by
\begin{equation*}
y_0^{(l,c)}=\frac{1}{2}\left(
0.42+\frac{\sum S_w^{(l)}}{10^4}
+\frac{\operatorname{mod}(\sum \overline P_{l,c},256)}{255}
\right),
\end{equation*}
where $S_w^{(l)}$ denotes the real component of the corresponding processed quaternion slice. 
The resulting 8-bit keystream is defined as
\begin{equation*}
K_{l,c}(t)=\operatorname{mod}\!\left(\left\lfloor10^{15}y_t\right\rfloor,256\right).
\end{equation*}
Then the component is masked by the Logistic keystream:
\begin{equation*}
C_{l,c}^{(0)}(t)=\overline P_{l,c}(t)\oplus K_{l,c}(t),
\end{equation*}
where $\oplus$ denotes the bitwise XOR operation.

The Logistic keystream masking modifies the component values but does not permute their spatial locations. An invertible Arnold permutation is therefore subsequently applied to provide spatial scrambling.

\subsubsection{Arnold Permutation and Core Component Recovery}

For zero-based coordinates, let $\Pi_{a,b}$ denote the Arnold permutation
\[
\begin{pmatrix}
u'\\
v'
\end{pmatrix}
=
\begin{pmatrix}
1 & a\\
b & ab+1
\end{pmatrix}
\begin{pmatrix}
u\\
v
\end{pmatrix}
\pmod n,
\]
where $(u,v)$ and $(u',v')$ denote the zero-based coordinates before and after
permutation, respectively, and $a,b$ are integer parameters determining the
Arnold transformation.

In the present implementation, $a=2$, $b=1$, and 20 iterations are used. 
The transmitted component is given by
\[
C_{l,c}
=
\Pi_{2,1}^{20}\left(C_{l,c}^{(0)}\right).
\]

During decryption, the inverse permutation $\Pi_{2,1}^{-20}$ is applied first,
followed by XOR decoding with the stored keystream and the inverse normalization
in Eq.~\eqref{normalization}. The recovered components are then
assembled into the core tensor $\widetilde{\mathcal S}$. For the QT-SVD branch,
the quaternion tensor is reconstructed as
\[
\widetilde{\mathcal A}
=
\mathcal U *_Q \widetilde{\mathcal S} *_Q \mathcal V^{*},
\]
whereas the T-SVD branch is reconstructed analogously using the T-product and
the corresponding adjoint.

In the adopted quaternion representation, the RGB channels of the original video are encoded in the three imaginary components. In the present implementation, encryption is applied to the three imaginary components of the core tensor. Any confidential information contained in the real component of the core tensor or other auxiliary reconstruction data must likewise be encrypted or transmitted through a secure channel.

\subsection{Experimental Setup and Evaluation Metrics}

The experiments are conducted on five sequences from the DAVIS 2017 TrainVal
480p dataset: \textit{blackswan}, \textit{bmx-trees},
\textit{car-roundabout}, \textit{dance-twirl}, and \textit{dog}, containing
50, 80, 75, 90, and 60 frames, respectively, for a total of 355 frames.
All available frames are used. Each frame is converted to double precision,
normalized to $[0,1]$, and resized to $256\times256$.

The QT-SVD and T-SVD branches use the same input data, preprocessing procedure,
and encryption parameters. In both cases, the full decomposition is employed
without singular-value truncation. The main data and implementation settings
are summarized in Table~\ref{tab:settings}.

All reconstruction metrics are computed directly from the reconstructed RGB frames in double precision after clipping to \([0,1]\) and before any visualization or video encoding. Consequently, neither image quantization nor lossy video compression affects the reported PSNR and SSIM values.

The evaluation considers four aspects: numerical decomposition and
reconstruction accuracy, statistical properties of the ciphertext,
differential sensitivity, and stability under perturbations of the input
video. The corresponding metrics are introduced below.

\begin{table}[!htbp]
\centering\small
\caption{Data and implementation settings.}
\label{tab:settings}
\setlength{\tabcolsep}{4pt}
\begin{tabular}{lp{.64\linewidth}}
\toprule
Item & Setting \\
\midrule
Dataset & DAVIS 2017 TrainVal, 480p JPEG sequences \\
Sequences (frames) & blackswan (50), bmx-trees (80), car-roundabout (75), dance-twirl (90), dog (60); 355 frames in total \\
Input preprocessing & Each frame is resized to \(256\times256\); RGB values are represented in double precision and normalized to \([0,1]\) \\
Quaternion basis & $\mathbf i,\mathbf j,\mathbf k$; pure quaternion RGB input \\
Decompositions & Full QT-SVD and full T-SVD; no singular-value truncation \\
Component scaling & Slice-wise, component-wise min--max scaling to 8-bit unsigned integers (\texttt{uint8}) \\
Logistic-map masking & $\mu=3.97$, offset $0.42$, scale $10^{15}$; data-dependent initial states \\
Arnold permutation & $a=2$, $b=1$, 20 iterations \\
Visualization frames & Frames 1, 25, and 50; comparison of the two methods at frame 25 \\
Statistical evaluation & Frame 25; 256 histogram bins; 3,000 horizontally adjacent R-channel pixel pairs \\
Input perturbations & Salt-and-pepper density $0.01$; Gaussian mean 0, variance $0.01$; centered \(12\times12\) zero-valued occlusion \\
Environment and seed & MATLAB R2022b; random seed 20260716 (\texttt{twister}) for paired perturbation experiments \\
\bottomrule
\end{tabular}
\end{table}

\subsubsection{Decomposition Errors}

For each method, let $\mathcal A_{\rm dec}$ denote the tensor reconstructed
directly from the corresponding decomposition factors. We report the absolute
Frobenius error
\[
e_{\rm dec}
=
\|\mathcal A-\mathcal A_{\rm dec}\|_F.
\]
This tensor-level error is evaluated before conversion to the RGB domain and
is not normalized by $\|\mathcal A\|_F$. It measures the numerical accuracy
of the corresponding decomposition.

\subsubsection{Reconstruction Quality Metrics}

The mean squared error (MSE), peak signal-to-noise ratio (PSNR), and
structural similarity index measure (SSIM) are used to evaluate the
reconstruction quality. For frame $l$, let $X_{ijcl}$ and
$\widetilde{X}_{ijcl}$ denote the reference and reconstructed RGB values,
respectively, at spatial position $(i,j)$ and color channel
$c\in\{1,2,3\}$. For an $M\times N$ frame, the frame-level MSE is defined by
\[
\operatorname{MSE}_l
=
\frac{1}{3MN}
\sum_{c=1}^{3}
\sum_{i=1}^{M}
\sum_{j=1}^{N}
\left(
X_{ijcl}-\widetilde{X}_{ijcl}
\right)^2.
\]
Since the RGB values are evaluated in $[0,1]$, the corresponding
PSNR~\cite{PSNRSSIM} is
\[
\operatorname{PSNR}_l
=
10\log_{10}
\left(
\frac{1}{\operatorname{MSE}_l}
\right).
\]
Thus, the errors from the three RGB channels are pooled before the logarithm
is taken, rather than averaging three channel-wise PSNR values.

For two corresponding color-channel images $X$ and $\widetilde{X}$, the local
SSIM comparison~\cite{Wang2004} has the form
\[
\operatorname{SSIM}(X,\widetilde{X})
=
\frac{
(2\mu_X\mu_{\widetilde{X}}+C_1)
(2\sigma_{X\widetilde{X}}+C_2)
}{
(\mu_X^2+\mu_{\widetilde{X}}^2+C_1)
(\sigma_X^2+\sigma_{\widetilde{X}}^2+C_2)
}.
\]
With dynamic range $L=1$ and the standard choices
$K_1=0.01$ and $K_2=0.03$, the constants are
\[
C_1=(K_1L)^2=0.01^2,
\qquad
C_2=(K_2L)^2=0.03^2.
\]
SSIM is computed separately for the three RGB channels using MATLAB
\texttt{ssim} with \texttt{DynamicRange=1}, and the resulting channel-wise
values are averaged to obtain the frame-level SSIM.

For each video, the reported PSNR and SSIM values are obtained by averaging the corresponding frame-level scores over all frames. The PSNR gain is defined as the difference between the QT-SVD and T-SVD values. The five-video mean is then obtained by averaging the corresponding video-level scores with equal weight, irrespective of the number of frames in each video.

\subsubsection{Ciphertext Statistics and Differential Sensitivity}

Histograms are computed using 256 bins for frame 25. For the adjacent-pixel
correlation analysis, 3,000 horizontally adjacent pixel pairs are sampled
from the R channel of the same frame, and the Pearson correlation coefficient
is computed. The plaintext and ciphertext use the same sampled coordinates
within each run.

For the differential sensitivity test, only the R-channel value of the first
pixel in the first input frame is perturbed according to
\[
R_{111}^{\prime}
=
\operatorname{mod}(R_{111}+0.01,1).
\]
The decomposition and encryption procedures are then repeated using the
perturbed input. Thus, the perturbation is introduced at the input-video
level, rather than directly into the ciphertext or the core tensor.

For each RGB channel, let $C^{(1)}$ and $C^{(2)}$ denote the ciphertext arrays
obtained from the original and perturbed inputs, respectively. With
$N=n^2p$ denoting the total number of ciphertext elements in the considered
channel, define
\[
D(i,j,l)
=
\begin{cases}
1, & C^{(1)}(i,j,l)\neq C^{(2)}(i,j,l),\\
0, & C^{(1)}(i,j,l)=C^{(2)}(i,j,l).
\end{cases}
\]
The number of pixel change rate (NPCR) and unified average changing intensity
(UACI)~\cite{NCPRUACI} are given by
\[
\operatorname{NPCR}
=
\frac{100}{N}
\sum_{i,j,l}
D(i,j,l)\%,
\]
and
\[
\operatorname{UACI}
=
\frac{100}{255N}
\sum_{i,j,l}
\left|
C^{(1)}(i,j,l)-C^{(2)}(i,j,l)
\right|\%.
\]
The values are computed separately for the three RGB channels and then
averaged.

For two independent arrays uniformly distributed over
$\{0,\ldots,255\}$, the expected reference values are
\[
\operatorname{NPCR}_{\rm ref}
=
100\frac{255}{256}
\approx 99.6094\%,
\qquad
\operatorname{UACI}_{\rm ref}
=
100\frac{257}{768}
\approx 33.4635\%.
\]
These values are used as reference levels for interpreting the observed
differential sensitivity.

Since the initial states of the Logistic sequences depend on the
input-derived core tensor, perturbing the input may also alter the generated
keystreams. Hence, the reported NPCR and UACI characterize the sensitivity of
the complete implemented procedure.

\subsection{Reconstruction Accuracy and Visual Comparison}

\subsubsection{Decomposition Accuracy}
Table~\ref{tab:errors} reports the absolute Frobenius errors for the
QT-SVD and T-SVD decompositions. The QT-SVD errors range from
$1.459\times10^{-11}$ to $2.227\times10^{-11}$, while the T-SVD errors
range from $3.446\times10^{-12}$ to $1.054\times10^{-11}$.
Both decompositions therefore reproduce the input tensors with high
numerical accuracy.

\begin{table}[!htbp]
\centering\small
\caption{Absolute Frobenius errors of the decompositions.}
\label{tab:errors}
\setlength{\tabcolsep}{4pt}
\begin{tabular}{lrr}
\toprule
\multirow{2}{*}{Video}
& \multicolumn{2}{c}{$e_{\rm dec}$} \\
\cmidrule(lr){2-3}
& QT-SVD & T-SVD \\
\midrule
blackswan & $1.459\times10^{-11}$ & $8.149\times10^{-12}$ \\
bmx-trees & $1.474\times10^{-11}$ & $3.911\times10^{-12}$ \\
car-roundabout & $1.758\times10^{-11}$ & $3.446\times10^{-12}$ \\
dance-twirl & $1.735\times10^{-11}$ & $4.317\times10^{-12}$ \\
dog & $2.227\times10^{-11}$ & $1.054\times10^{-11}$ \\
\bottomrule
\end{tabular}
\end{table}

\subsubsection{Reconstruction Quality}

Table~\ref{tab:reconstruction} summarizes the video-averaged PSNR and SSIM
values obtained after core encryption and recovery.

\begin{table}[!htbp]
\centering\small
\caption{Reconstruction quality after core encryption and recovery.}
\label{tab:reconstruction}
\setlength{\tabcolsep}{4pt}
\begin{tabular}{lrrrrr}
\toprule
Video & \multicolumn{2}{c}{PSNR (dB)} & \multicolumn{2}{c}{SSIM} & Gain (dB) \\
\cmidrule(lr){2-3}
\cmidrule(lr){4-5}
\cmidrule(lr){6-6}
& QT-SVD & T-SVD & QT-SVD & T-SVD & $\Delta$PSNR \\
\midrule
blackswan & 67.4470 & 62.8259 & 0.999952 & 0.999864 & 4.6211 \\
bmx-trees & 57.3380 & 41.0254 & 0.999527 & 0.980869 & 16.3126 \\
car-roundabout & 59.3204 & 46.7411 & 0.999625 & 0.993376 & 12.5793 \\
dance-twirl & 68.3103 & 48.9359 & 0.999953 & 0.996055 & 19.3744 \\
dog & 72.5327 & 46.5558 & 0.999986 & 0.994785 & 25.9769 \\
\midrule
Five-video mean & 64.9897 & 49.2168 & 0.999809 & 0.992990 & 15.7730 \\
\bottomrule
\end{tabular}
\end{table}
QT-SVD yields higher PSNR and SSIM than T-SVD for all five test sequences.
The five-video mean PSNR is $64.9897$~dB for QT-SVD and $49.2168$~dB for
T-SVD, while the corresponding mean SSIM values are $0.999809$ and
$0.992990$, respectively. The PSNR difference ranges from $4.6211$~dB for
\textit{blackswan} to $25.9769$~dB for \textit{dog}, with a five-video mean
of $15.7730$~dB. Since the SSIM values for both methods are already close to
one, only the PSNR difference is included as an additional comparison measure
in Table~\ref{tab:reconstruction}.

\subsubsection{Visual Comparison}

Figures~\ref{fig:qt1}--\ref{fig:qt50} show the QT-SVD results at frames
1, 25, and 50. In each figure, the five columns correspond to
\textit{blackswan}, \textit{bmx-trees}, \textit{car-roundabout},
\textit{dance-twirl}, and \textit{dog}, while the three rows show the original
RGB frames, encrypted-core visualizations, and reconstructed RGB frames. The reconstructed frames preserve the scene content, whereas the encrypted
cores exhibit noise-like patterns.

Figure~\ref{fig:comparison} compares the QT-SVD and T-SVD results at frame~25 for all five sequences. Both methods preserve the visual content of the original frames, while their reconstruction differences are small at the displayed scale. Therefore, Table~\ref{tab:reconstruction} reports the quantitative comparisons computed from the double-precision reconstructed RGB frames.

\begin{figure}[!htbp]
\centering
\begingroup\setlength{\tabcolsep}{0pt}\renewcommand{\arraystretch}{0}
\begin{tabular}{@{}ccccc@{}}
\makebox[.17\linewidth]{\scriptsize blackswan} & \makebox[.17\linewidth]{\scriptsize bmx-trees} & \makebox[.17\linewidth]{\scriptsize car-roundabout} & \makebox[.17\linewidth]{\scriptsize dance-twirl} & \makebox[.17\linewidth]{\scriptsize dog} \\[3pt]
\includegraphics[width=.18\linewidth]{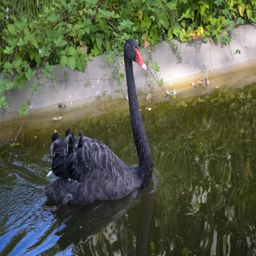} & \includegraphics[width=.18\linewidth]{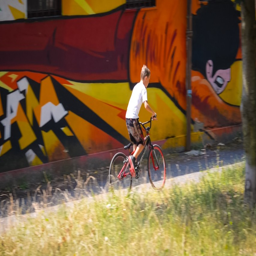} & \includegraphics[width=.18\linewidth]{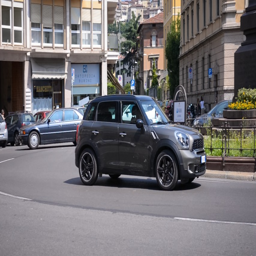} & \includegraphics[width=.18\linewidth]{figures/QTSVD/dance-twirl_plain_1.png} & \includegraphics[width=.18\linewidth]{figures/QTSVD/dog_plain_1.png} \\[-.2pt]
\includegraphics[width=.18\linewidth]{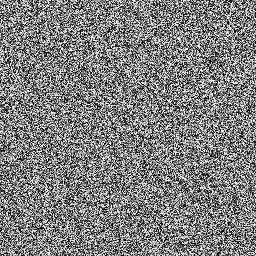} & \includegraphics[width=.18\linewidth]{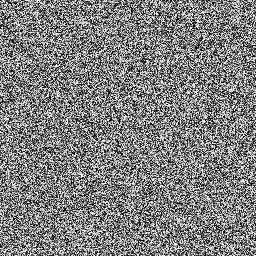} & \includegraphics[width=.18\linewidth]{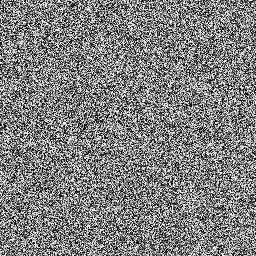} & \includegraphics[width=.18\linewidth]{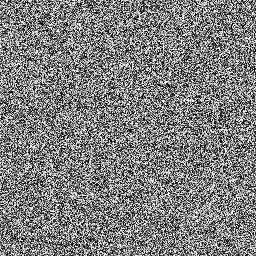} & \includegraphics[width=.18\linewidth]{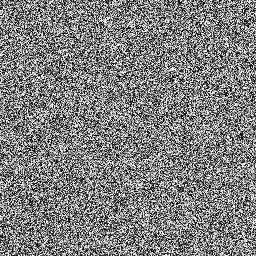} \\[-.2pt]
\includegraphics[width=.18\linewidth]{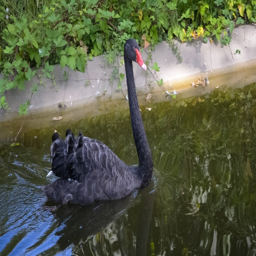} & \includegraphics[width=.18\linewidth]{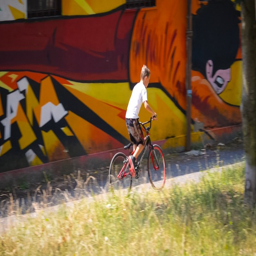} & \includegraphics[width=.18\linewidth]{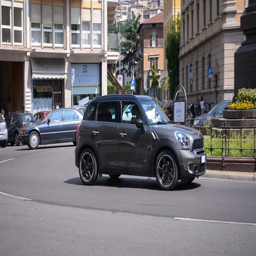} & \includegraphics[width=.18\linewidth]{figures/QTSVD/dance-twirl_decrypted_1.png} & \includegraphics[width=.18\linewidth]{figures/QTSVD/dog_decrypted_1.png} \\[-.2pt]
\end{tabular}\endgroup
\caption{QT-SVD results at frame 1.}
\label{fig:qt1}
\end{figure}
\begin{figure}[!htbp]
\centering
\begingroup\setlength{\tabcolsep}{0pt}\renewcommand{\arraystretch}{0}
\begin{tabular}{@{}ccccc@{}}
\makebox[.17\linewidth]{\scriptsize blackswan} & \makebox[.17\linewidth]{\scriptsize bmx-trees} & \makebox[.17\linewidth]{\scriptsize car-roundabout} & \makebox[.17\linewidth]{\scriptsize dance-twirl} & \makebox[.17\linewidth]{\scriptsize dog} \\[3pt]
\includegraphics[width=.18\linewidth]{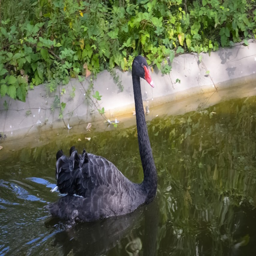} & \includegraphics[width=.18\linewidth]{figures/QTSVD/bmx-trees_plain_25.png} & \includegraphics[width=.18\linewidth]{figures/QTSVD/car-roundabout_plain_25.png} & \includegraphics[width=.18\linewidth]{figures/QTSVD/dance-twirl_plain_25.png} & \includegraphics[width=.18\linewidth]{figures/QTSVD/dog_plain_25.png} \\[-.2pt]
\includegraphics[width=.18\linewidth]{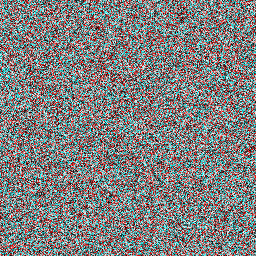} & \includegraphics[width=.18\linewidth]{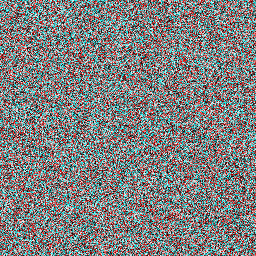} & \includegraphics[width=.18\linewidth]{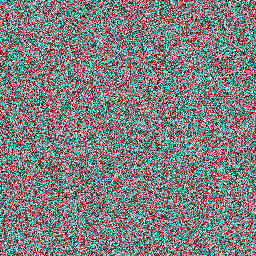} & \includegraphics[width=.18\linewidth]{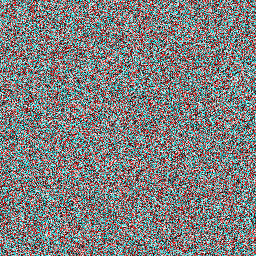} & \includegraphics[width=.18\linewidth]{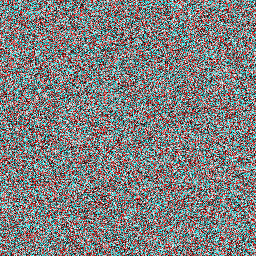} \\[-.2pt]
\includegraphics[width=.18\linewidth]{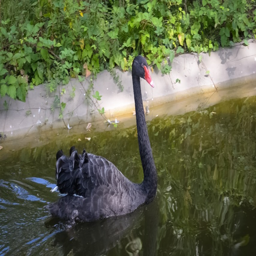} & \includegraphics[width=.18\linewidth]{figures/QTSVD/bmx-trees_decrypted_25.png} & \includegraphics[width=.18\linewidth]{figures/QTSVD/car-roundabout_decrypted_25.png} & \includegraphics[width=.18\linewidth]{figures/QTSVD/dance-twirl_decrypted_25.png} & \includegraphics[width=.18\linewidth]{figures/QTSVD/dog_decrypted_25.png} \\[-.2pt]
\end{tabular}\endgroup
\caption{QT-SVD results at frame 25.}
\label{fig:qt25}
\end{figure}
\begin{figure}[!htbp]
\centering
\begingroup\setlength{\tabcolsep}{0pt}\renewcommand{\arraystretch}{0}
\begin{tabular}{@{}ccccc@{}}
\makebox[.17\linewidth]{\scriptsize blackswan} & \makebox[.17\linewidth]{\scriptsize bmx-trees} & \makebox[.17\linewidth]{\scriptsize car-roundabout} & \makebox[.17\linewidth]{\scriptsize dance-twirl} & \makebox[.17\linewidth]{\scriptsize dog} \\[3pt]
\includegraphics[width=.18\linewidth]{figures/QTSVD/blackswan_plain_50.png} & \includegraphics[width=.18\linewidth]{figures/QTSVD/bmx-trees_plain_50.png} & \includegraphics[width=.18\linewidth]{figures/QTSVD/car-roundabout_plain_50.png} & \includegraphics[width=.18\linewidth]{figures/QTSVD/dance-twirl_plain_50.png} & \includegraphics[width=.18\linewidth]{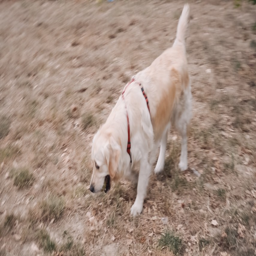} \\[-.2pt]
\includegraphics[width=.18\linewidth]{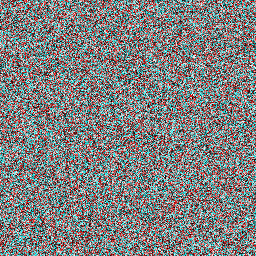} & \includegraphics[width=.18\linewidth]{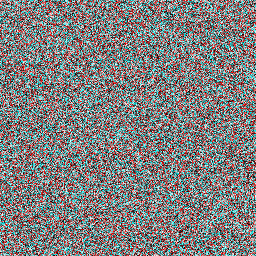} & \includegraphics[width=.18\linewidth]{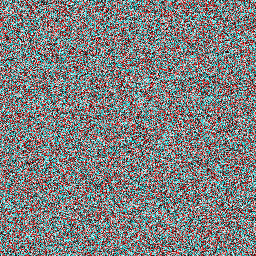} & \includegraphics[width=.18\linewidth]{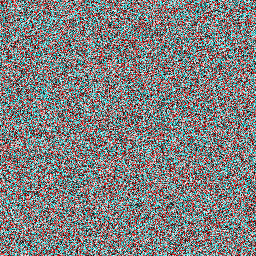} & \includegraphics[width=.18\linewidth]{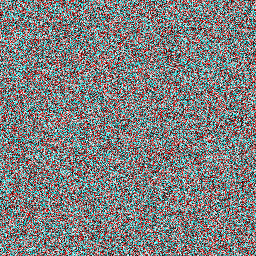} \\[-.2pt]
\includegraphics[width=.18\linewidth]{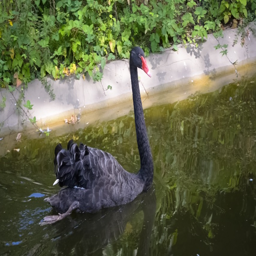} & \includegraphics[width=.18\linewidth]{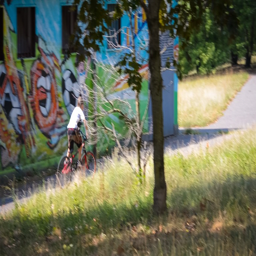} & \includegraphics[width=.18\linewidth]{figures/QTSVD/car-roundabout_decrypted_50.png} & \includegraphics[width=.18\linewidth]{figures/QTSVD/dance-twirl_decrypted_50.png} & \includegraphics[width=.18\linewidth]{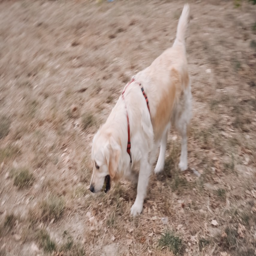} \\[-.2pt]
\end{tabular}\endgroup
\caption{QT-SVD results at frame 50.}
\label{fig:qt50}
\end{figure}

\clearpage
\begin{figure}[!htbp]
\centering
\begingroup\setlength{\tabcolsep}{0pt}\renewcommand{\arraystretch}{0}
\begin{tabular}{@{}ccccc@{}}
\makebox[.18\linewidth]{\scriptsize blackswan} & \makebox[.18\linewidth]{\scriptsize bmx-trees} & \makebox[.18\linewidth]{\scriptsize car-roundabout} & \makebox[.18\linewidth]{\scriptsize dance-twirl} & \makebox[.18\linewidth]{\scriptsize dog} \\[3pt]
\includegraphics[width=.18\linewidth]{figures/QTSVD/blackswan_plain_25.png} & \includegraphics[width=.18\linewidth]{figures/QTSVD/bmx-trees_plain_25.png} & \includegraphics[width=.18\linewidth]{figures/QTSVD/car-roundabout_plain_25.png} & \includegraphics[width=.18\linewidth]{figures/QTSVD/dance-twirl_plain_25.png} & \includegraphics[width=.18\linewidth]{figures/QTSVD/dog_plain_25.png} \\[-.2pt]
\includegraphics[width=.18\linewidth]{figures/QTSVD/blackswan_cipher_25.png} & \includegraphics[width=.18\linewidth]{figures/QTSVD/bmx-trees_cipher_25.png} & \includegraphics[width=.18\linewidth]{figures/QTSVD/car-roundabout_cipher_25.png} & \includegraphics[width=.18\linewidth]{figures/QTSVD/dance-twirl_cipher_25.png} & \includegraphics[width=.18\linewidth]{figures/QTSVD/dog_cipher_25.png} \\[-.2pt]
\includegraphics[width=.18\linewidth]{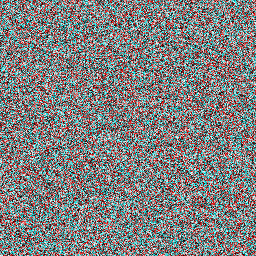} & \includegraphics[width=.18\linewidth]{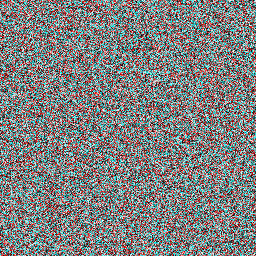} & \includegraphics[width=.18\linewidth]{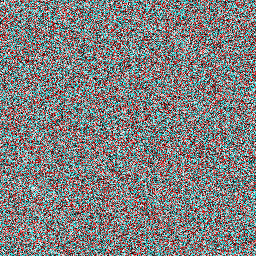} & \includegraphics[width=.18\linewidth]{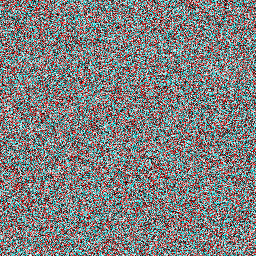} & \includegraphics[width=.18\linewidth]{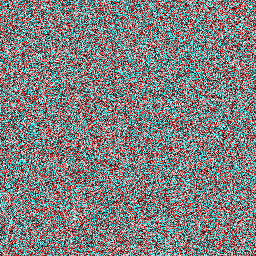} \\[-.2pt]
\includegraphics[width=.18\linewidth]{figures/QTSVD/blackswan_decrypted_25.png} & \includegraphics[width=.18\linewidth]{figures/QTSVD/bmx-trees_decrypted_25.png} & \includegraphics[width=.18\linewidth]{figures/QTSVD/car-roundabout_decrypted_25.png} & \includegraphics[width=.18\linewidth]{figures/QTSVD/dance-twirl_decrypted_25.png} & \includegraphics[width=.18\linewidth]{figures/QTSVD/dog_decrypted_25.png} \\[-.2pt]
\includegraphics[width=.18\linewidth]{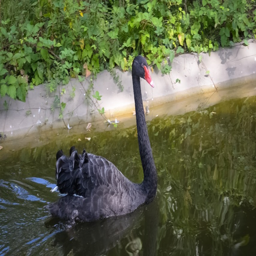} & \includegraphics[width=.18\linewidth]{figures/TSVD/bmx-trees_decrypted_25.png} & \includegraphics[width=.18\linewidth]{figures/TSVD/car-roundabout_decrypted_25.png} & \includegraphics[width=.18\linewidth]{figures/TSVD/dance-twirl_decrypted_25.png} & \includegraphics[width=.18\linewidth]{figures/TSVD/dog_decrypted_25.png} \\[-.2pt]
\end{tabular}\endgroup
\caption{QT-SVD and T-SVD results at frame~25. From top to bottom:
original RGB, QT-SVD encrypted core, T-SVD encrypted core,
QT-SVD reconstruction, and T-SVD reconstruction.}
\label{fig:comparison}
\end{figure}

\subsection{Statistical Analysis and Differential Sensitivity of the Encrypted Core}
\subsubsection{Histograms and Adjacent-Pixel Correlation}

Figure~\ref{fig:histograms} shows the histograms of the original RGB frames and the three encrypted-core components at frame~25 for the five test videos.
The original histograms vary with scene content, whereas the encrypted-core components exhibit substantially flatter intensity distributions. These histograms are used for qualitative comparison of the distribution shapes rather than as evidence of exact uniformity.

Figure~\ref{fig:correlations} compares the horizontal adjacent-pixel correlations in the R channel at frame~25. The correlation coefficients of
the original frames range from $0.920005$ to $0.980894$, indicating strong local dependence. After encryption, the corresponding QT-SVD coefficients
range from $-0.013116$ to $0.015532$, and the scatter plots become broadly dispersed, indicating a substantial reduction in horizontal adjacent-pixel correlation.

\begin{figure}[!htbp]
\centering

% ---------- First row: three videos ----------
\begin{minipage}[t]{0.31\linewidth}
\centering
{\scriptsize blackswan}\\[3pt]
\includegraphics[width=\linewidth]{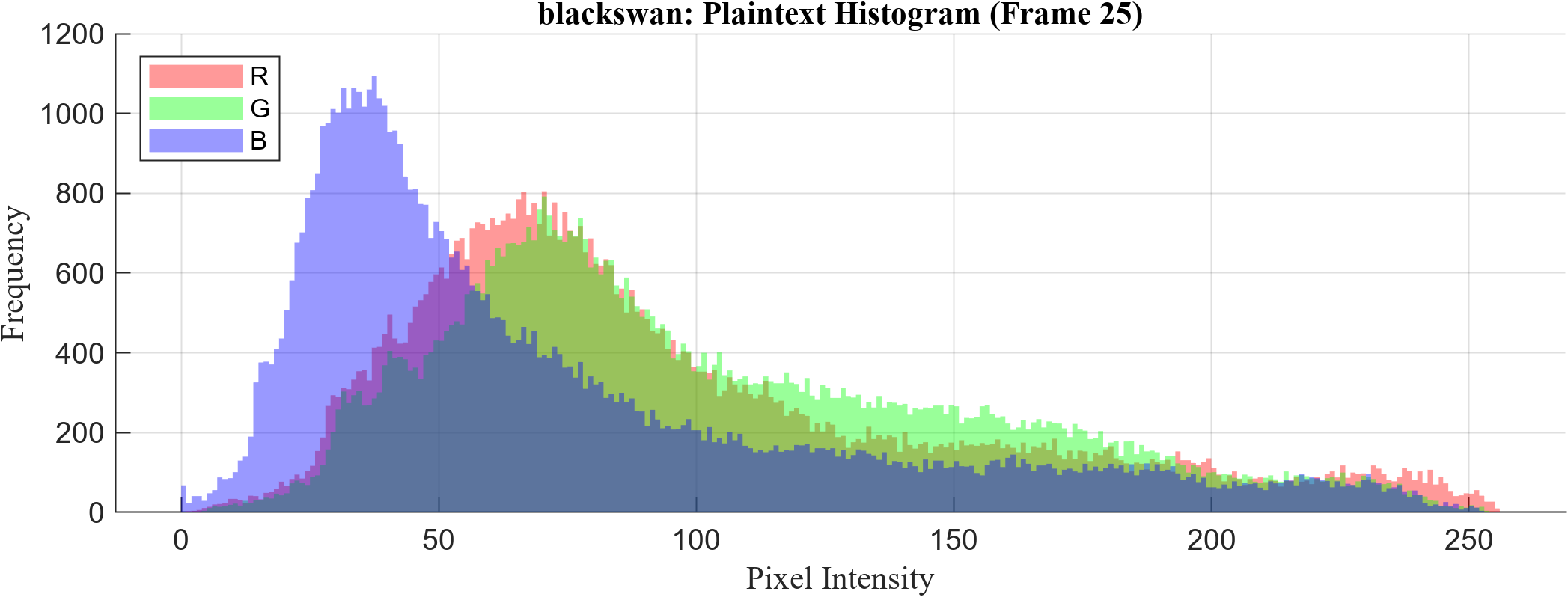}\\[-1pt]
\includegraphics[width=\linewidth]{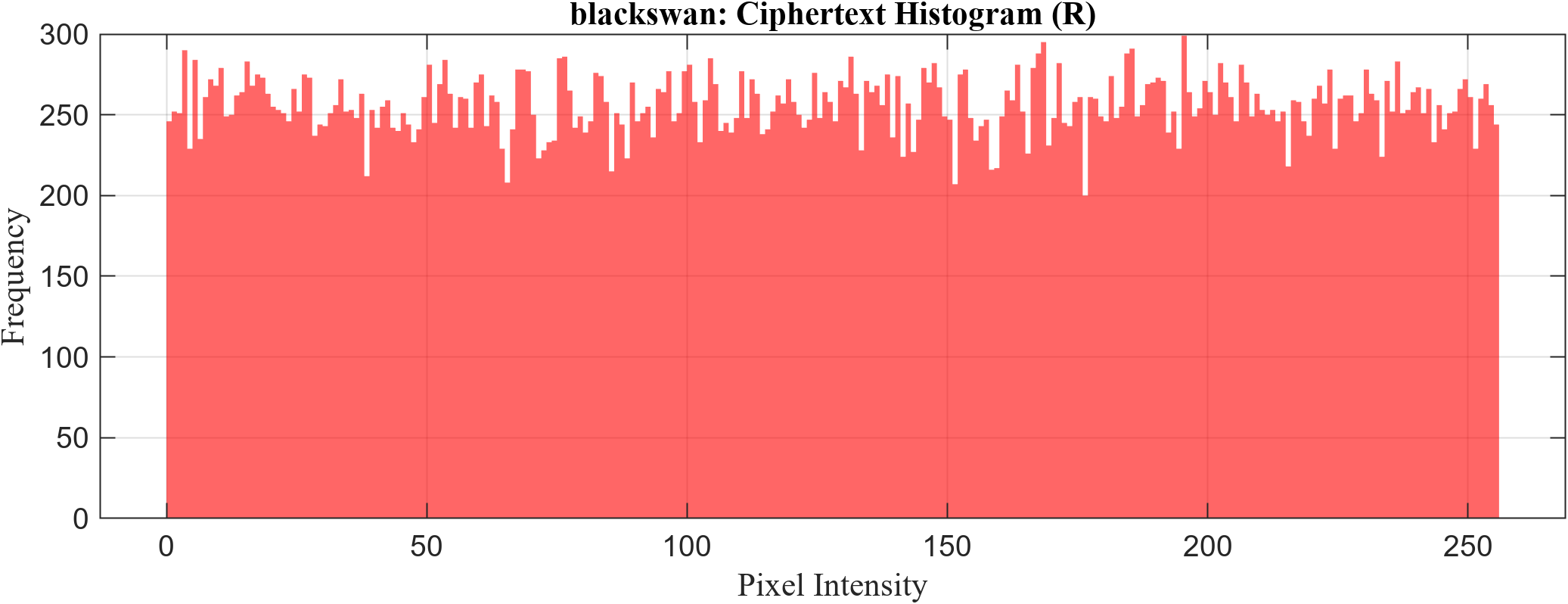}\\[-1pt]
\includegraphics[width=\linewidth]{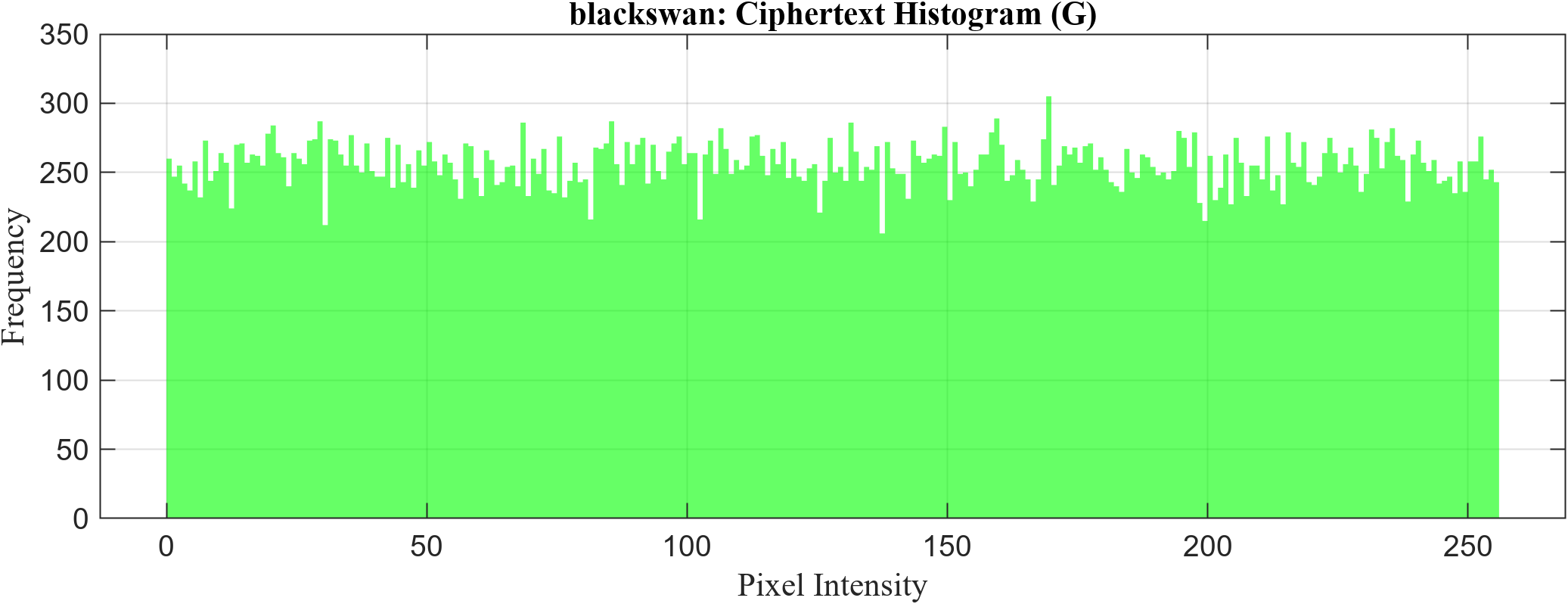}\\[-1pt]
\includegraphics[width=\linewidth]{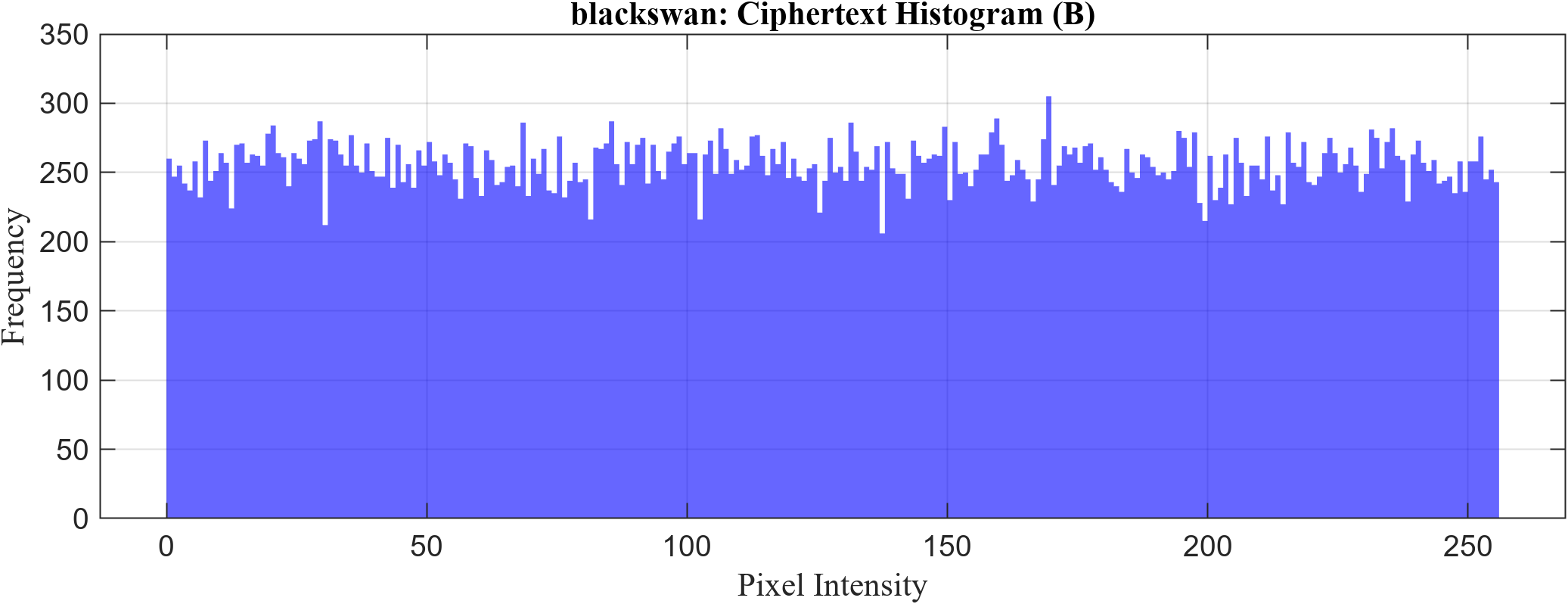}
\end{minipage}
\hfill
\begin{minipage}[t]{0.31\linewidth}
\centering
{\scriptsize bmx-trees}\\[3pt]
\includegraphics[width=\linewidth]{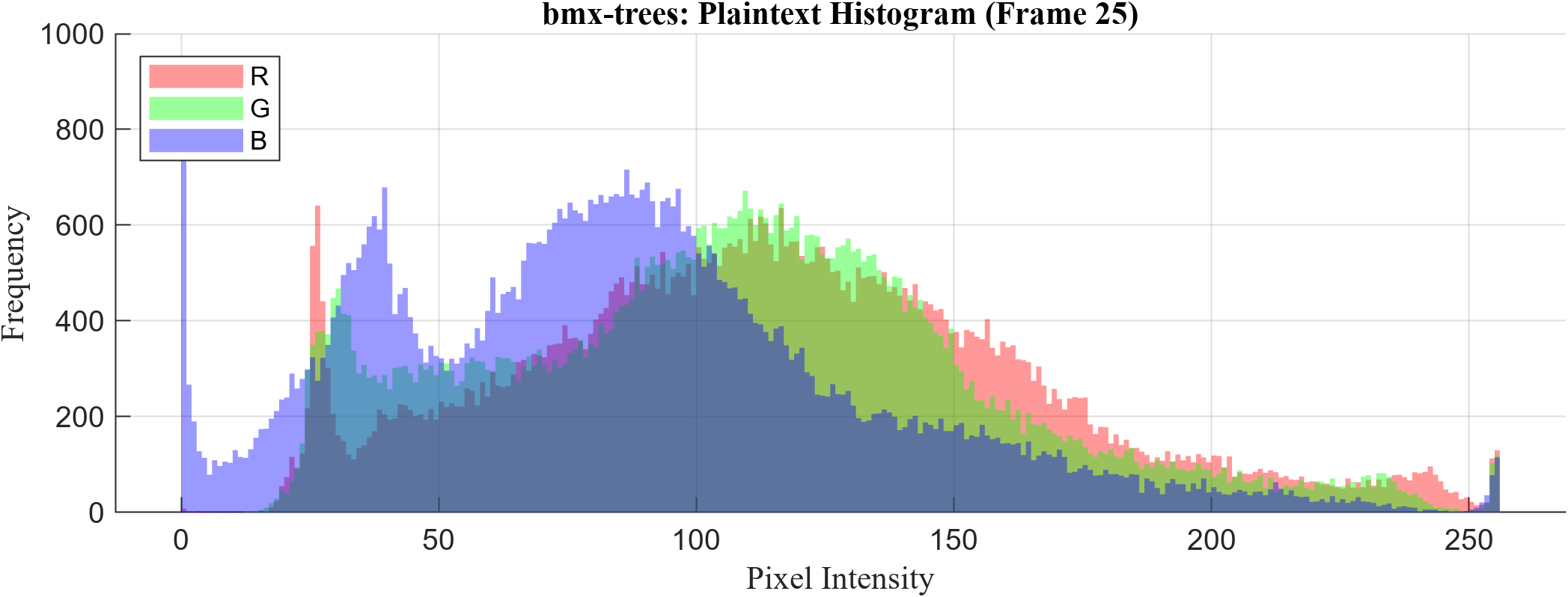}\\[-1pt]
\includegraphics[width=\linewidth]{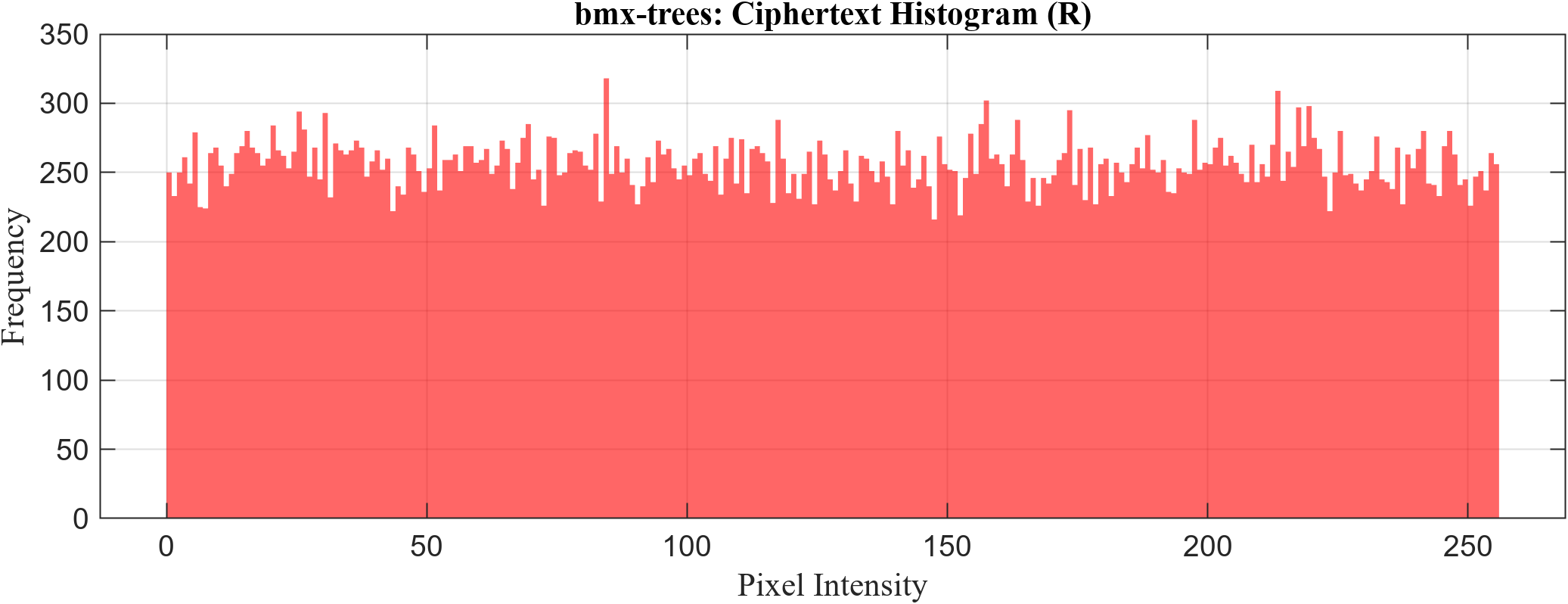}\\[-1pt]
\includegraphics[width=\linewidth]{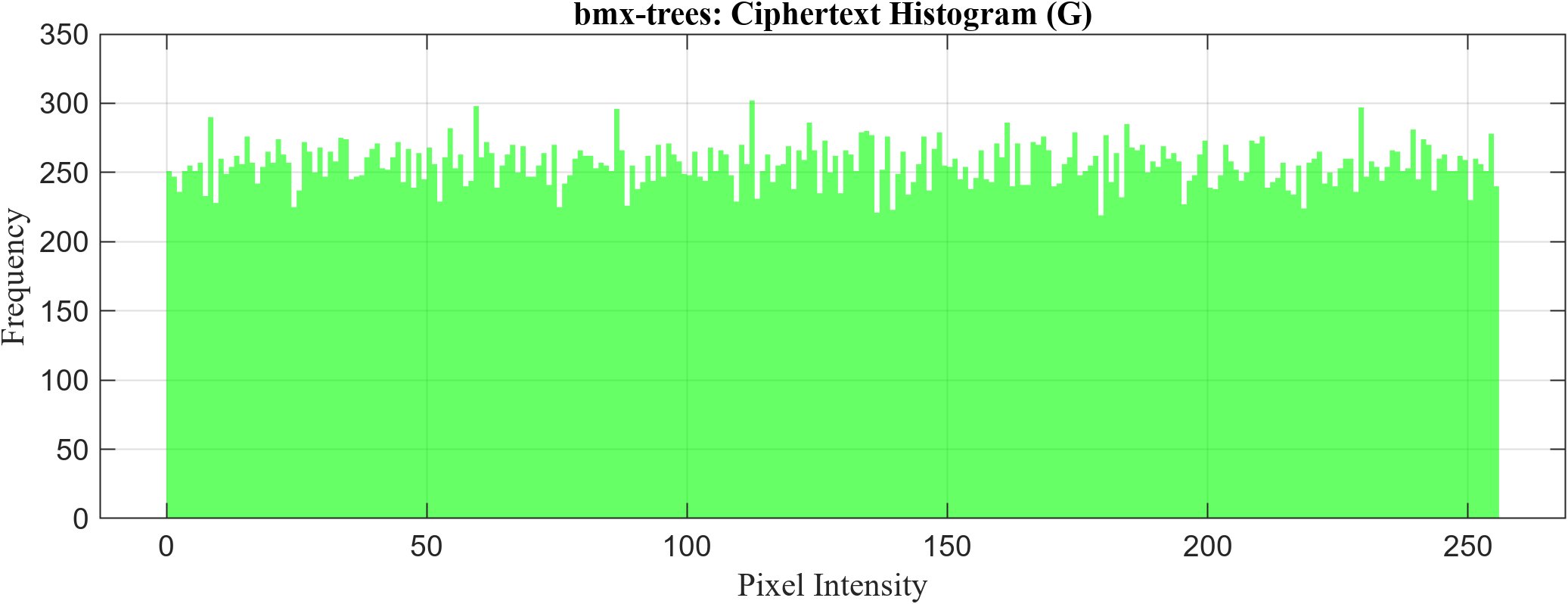}\\[-1pt]
\includegraphics[width=\linewidth]{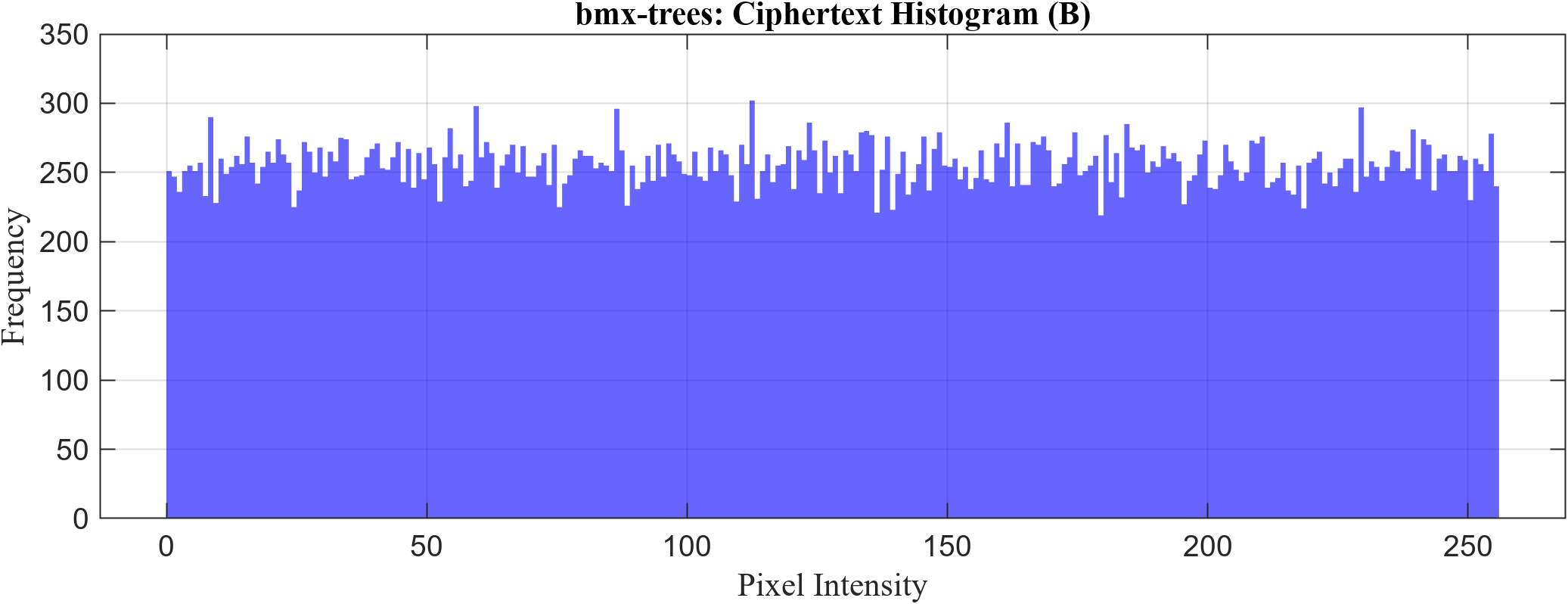}
\end{minipage}
\hfill
\begin{minipage}[t]{0.31\linewidth}
\centering
{\scriptsize car-roundabout}\\[3pt]
\includegraphics[width=\linewidth]{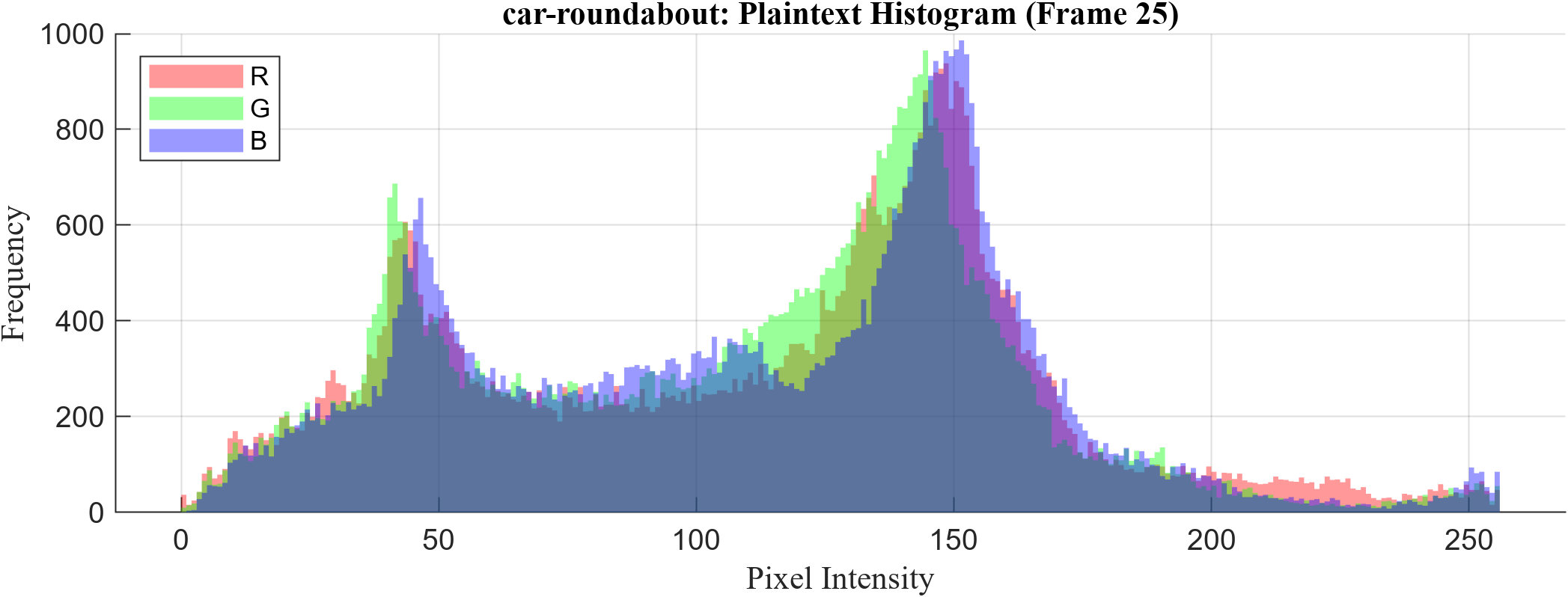}\\[-1pt]
\includegraphics[width=\linewidth]{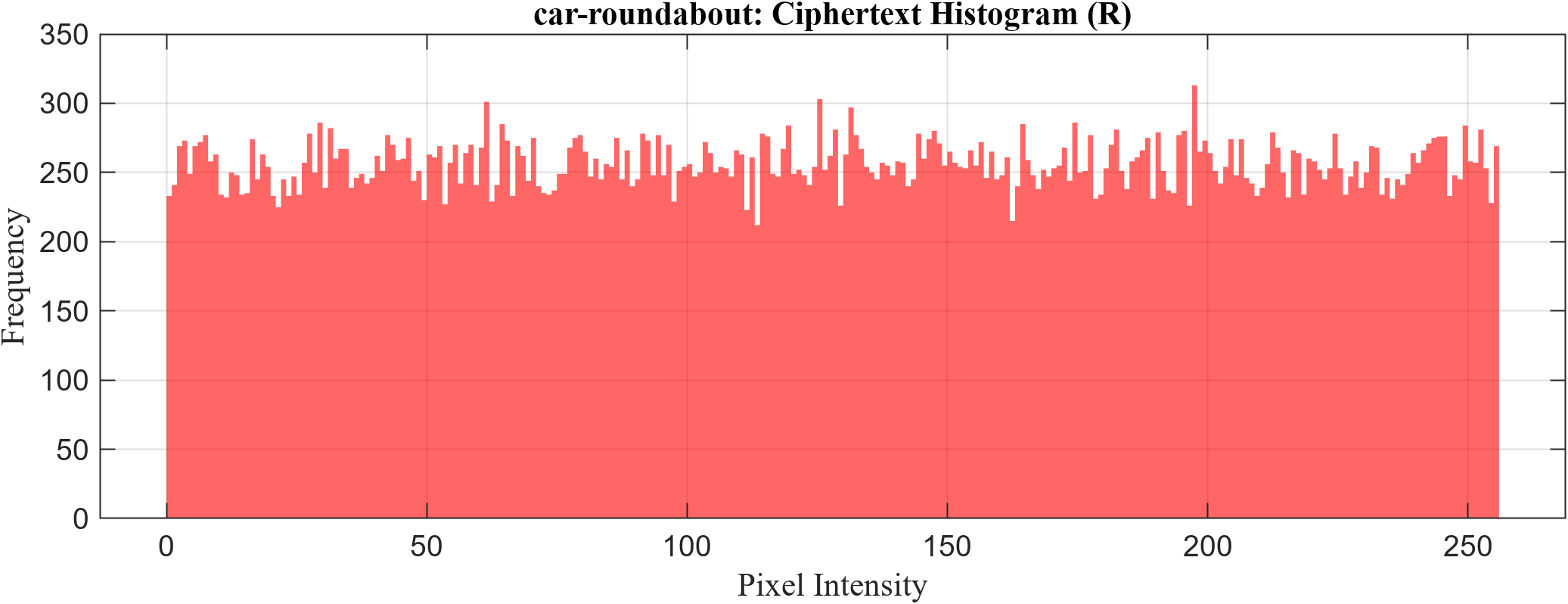}\\[-1pt]
\includegraphics[width=\linewidth]{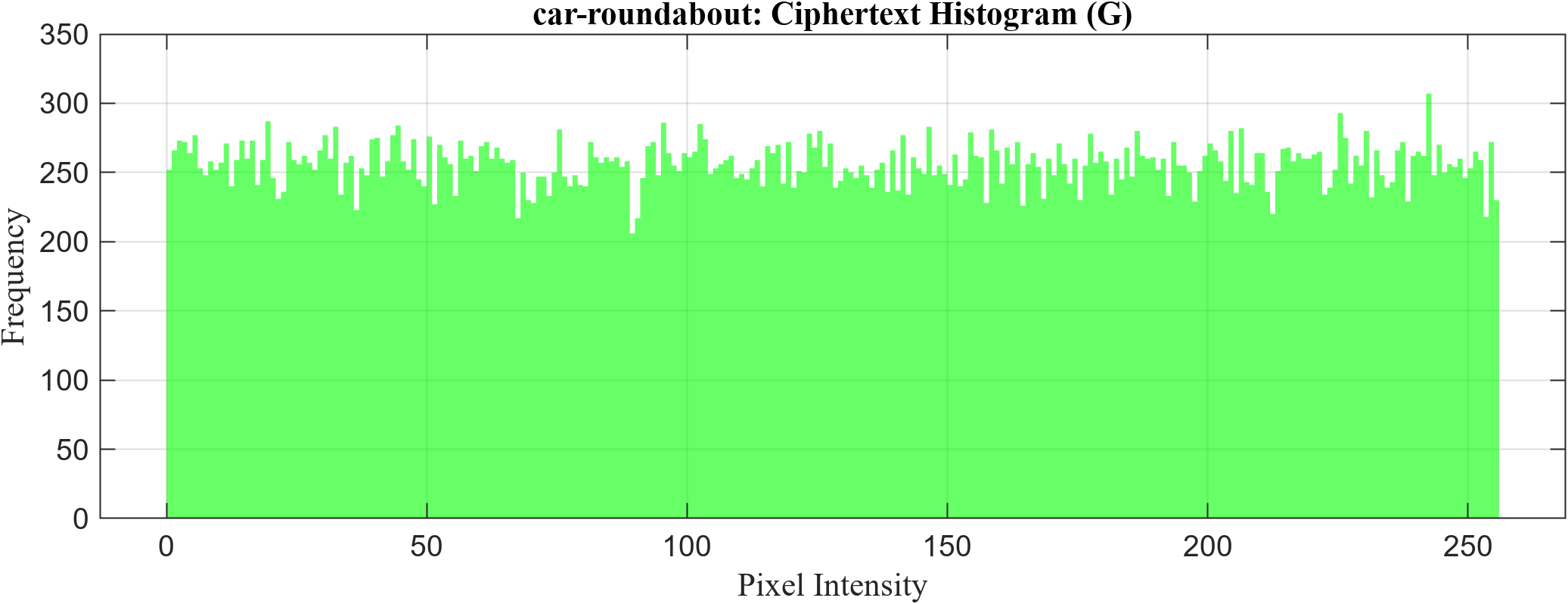}\\[-1pt]
\includegraphics[width=\linewidth]{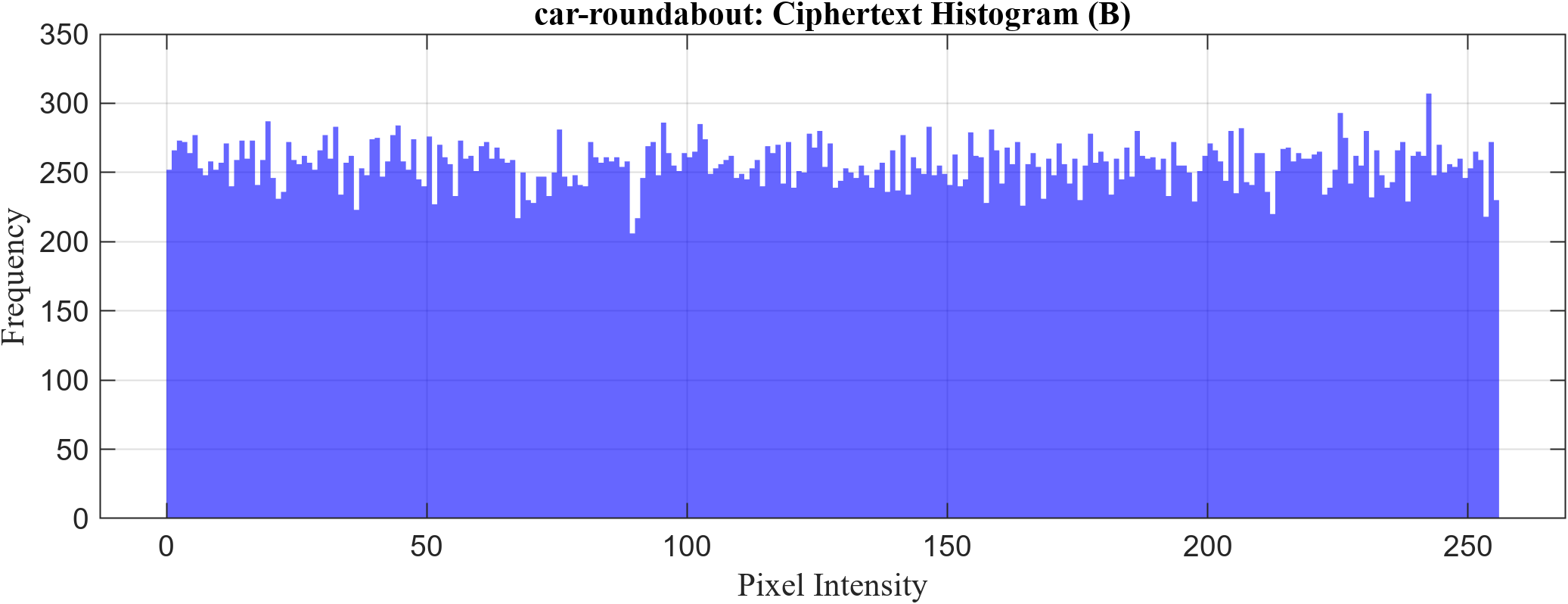}
\end{minipage}

\vspace{6pt}

% ---------- Second row: two videos ----------
\makebox[\linewidth][c]{%
\begin{minipage}[t]{0.31\linewidth}
\centering
{\scriptsize dance-twirl}\\[3pt]
\includegraphics[width=\linewidth]{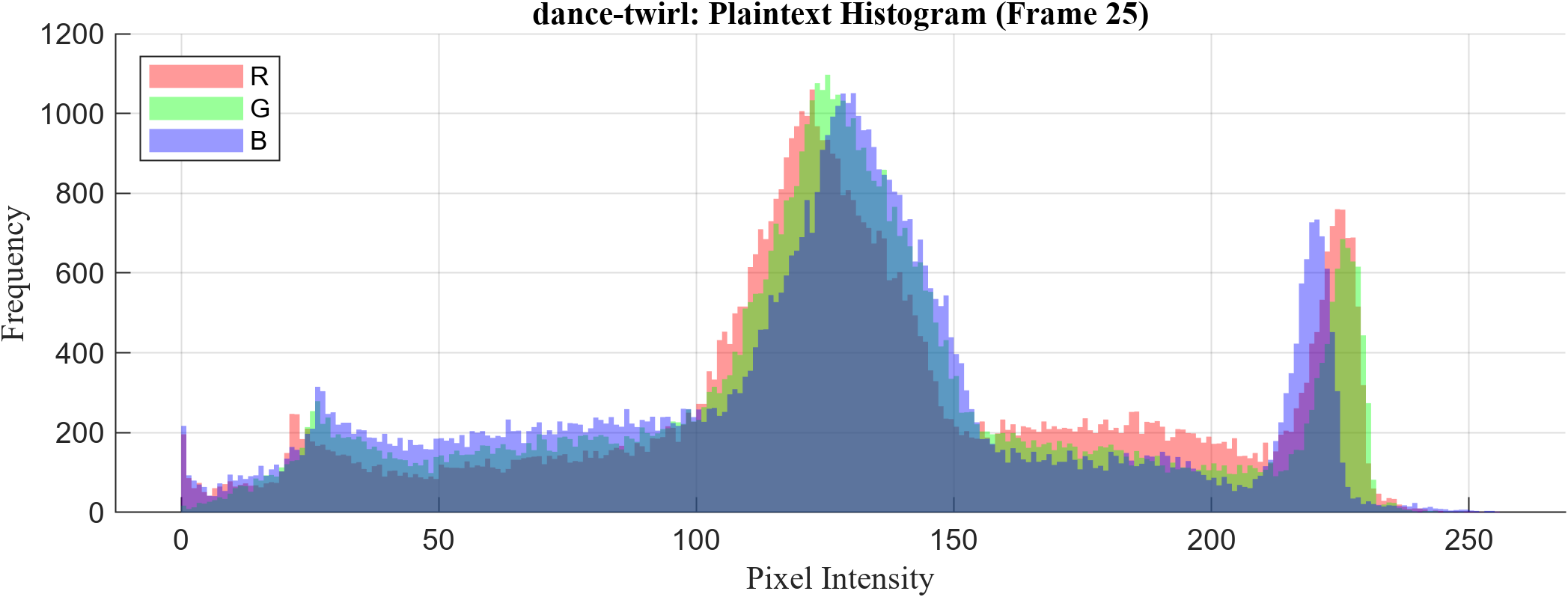}\\[-1pt]
\includegraphics[width=\linewidth]{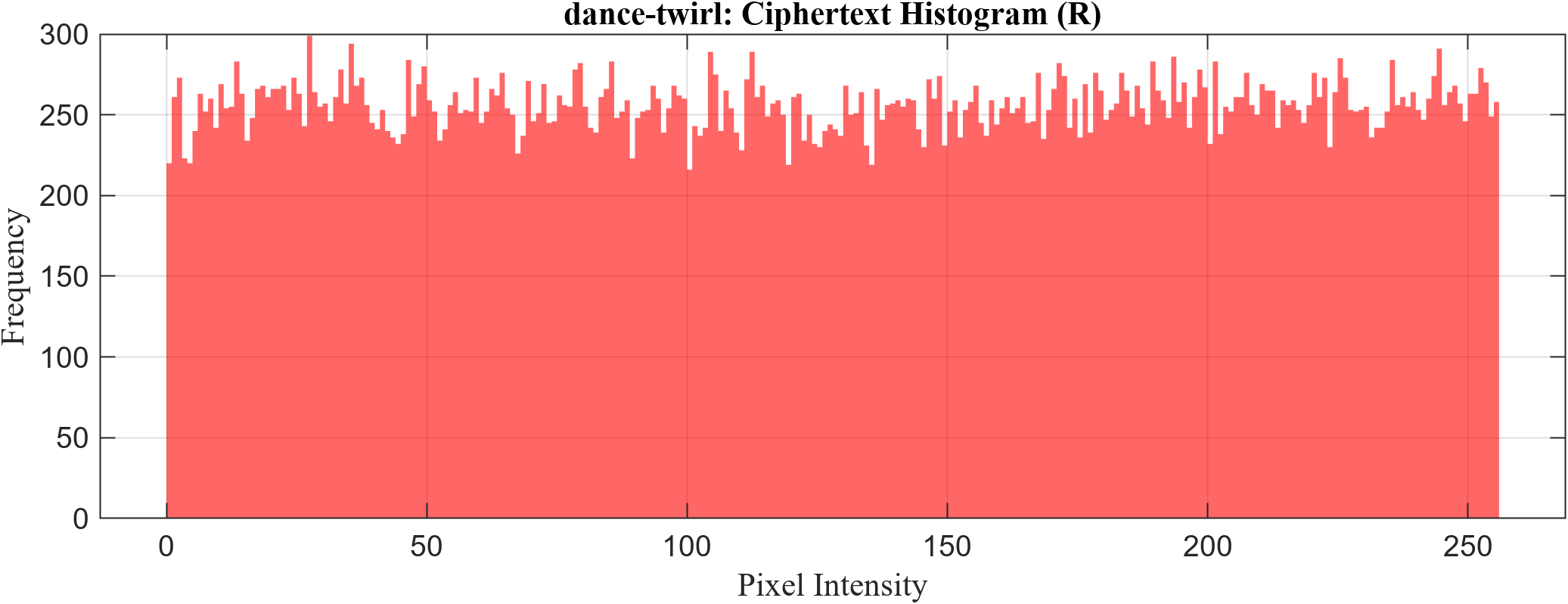}\\[-1pt]
\includegraphics[width=\linewidth]{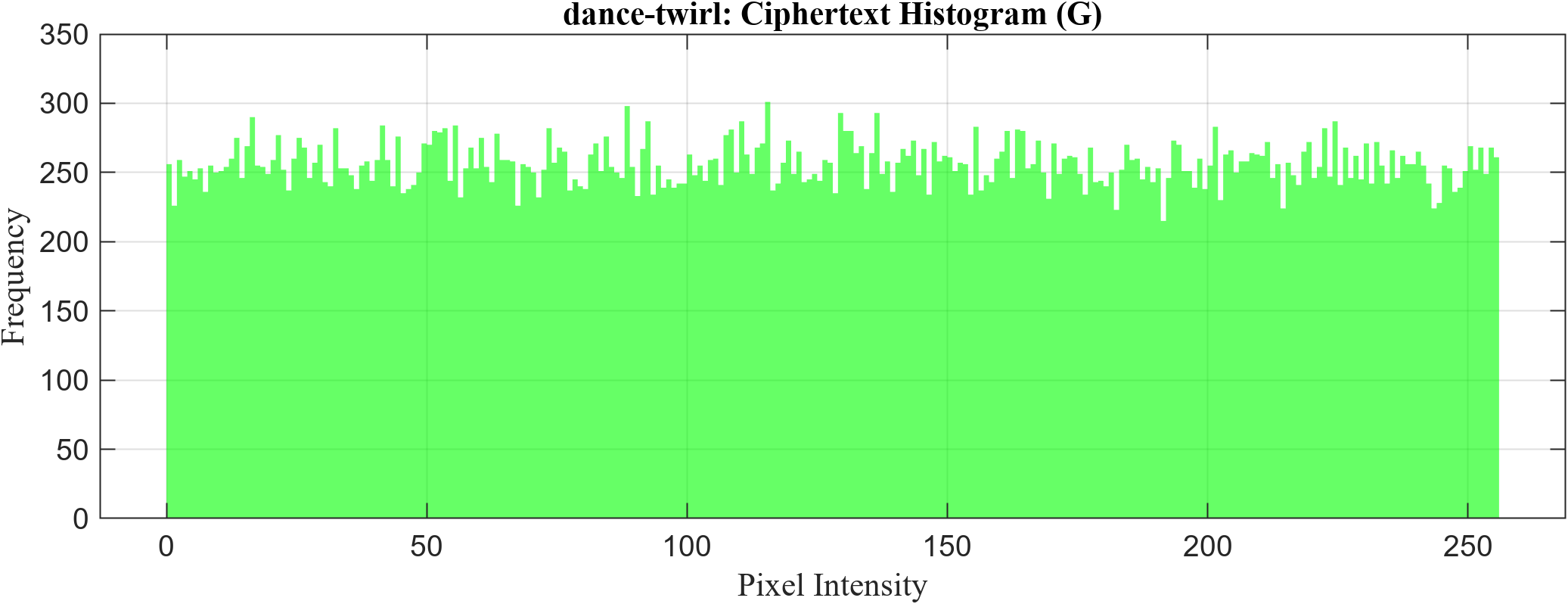}\\[-1pt]
\includegraphics[width=\linewidth]{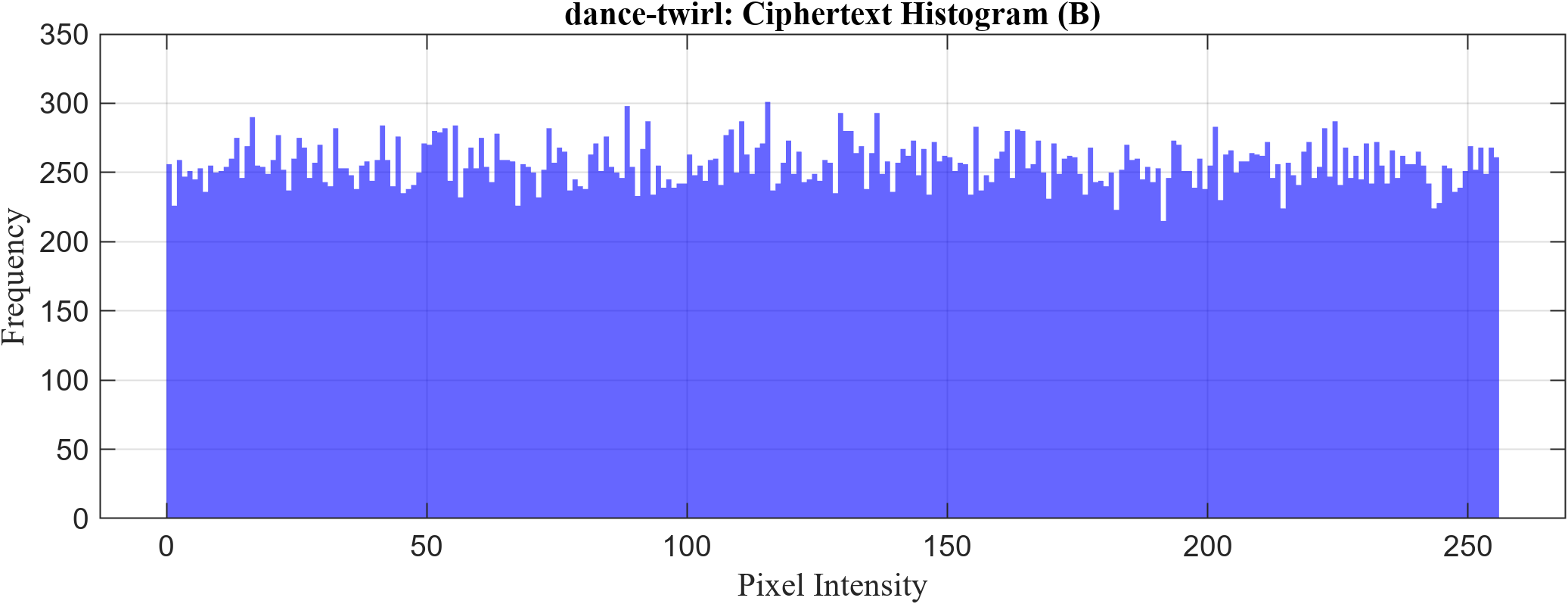}
\end{minipage}
\hspace{0.03\linewidth}
\begin{minipage}[t]{0.31\linewidth}
\centering
{\scriptsize dog}\\[3pt]
\includegraphics[width=\linewidth]{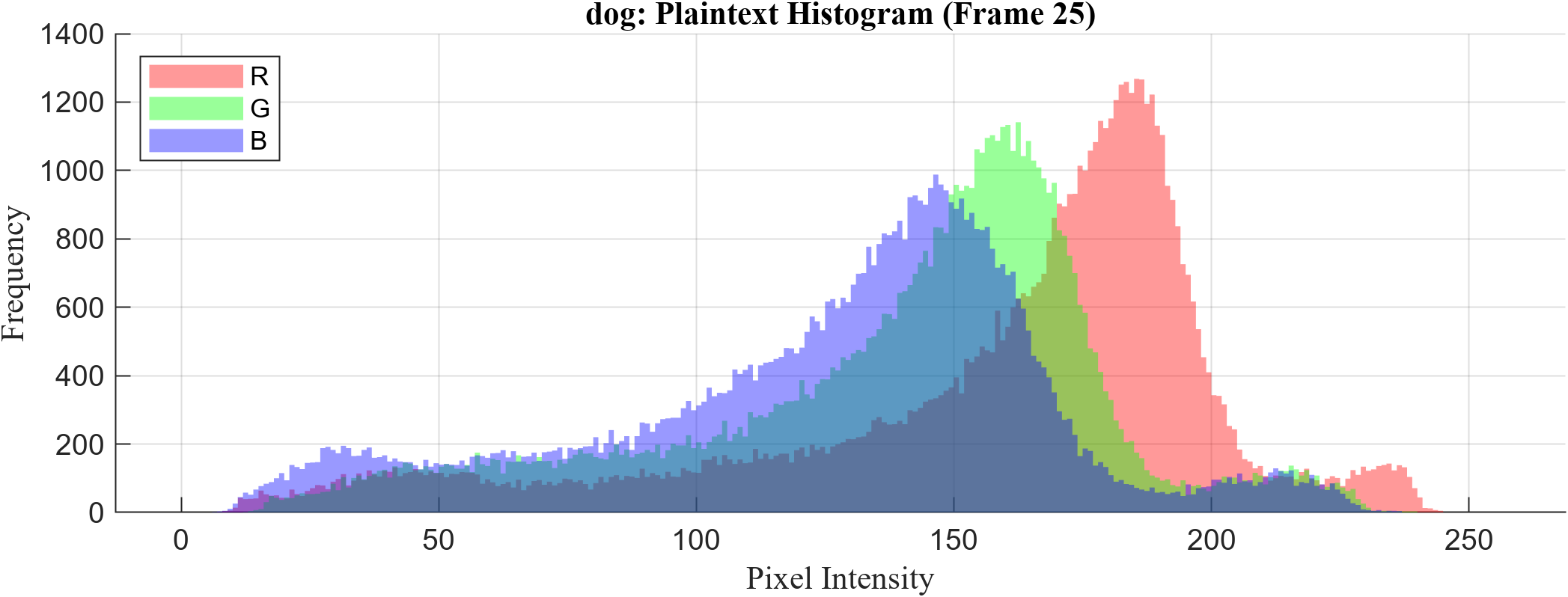}\\[-1pt]
\includegraphics[width=\linewidth]{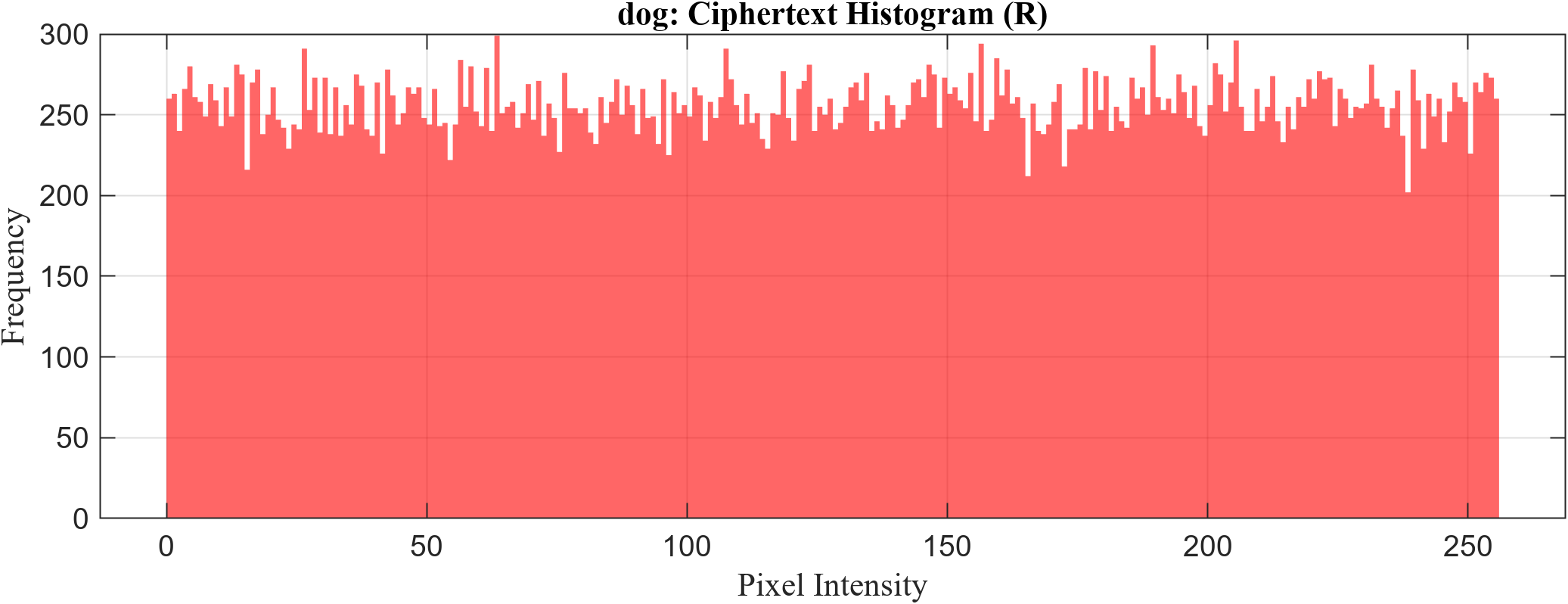}\\[-1pt]
\includegraphics[width=\linewidth]{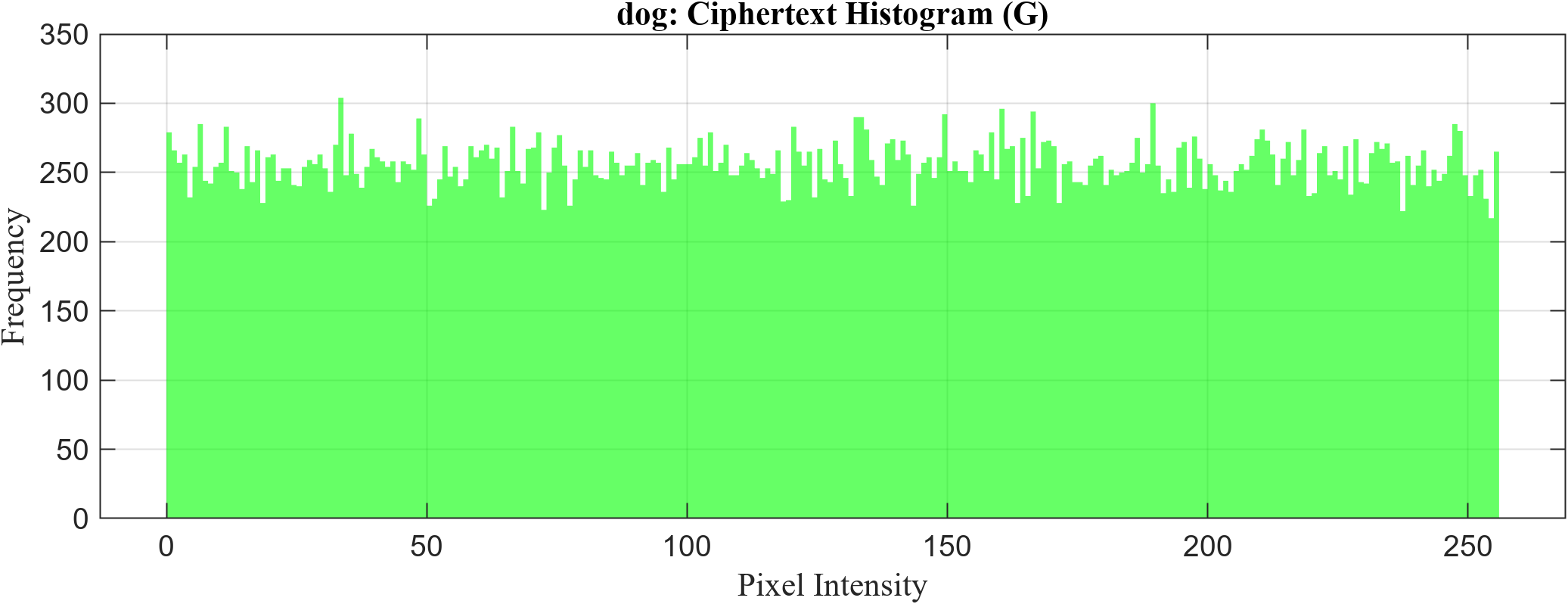}\\[-1pt]
\includegraphics[width=\linewidth]{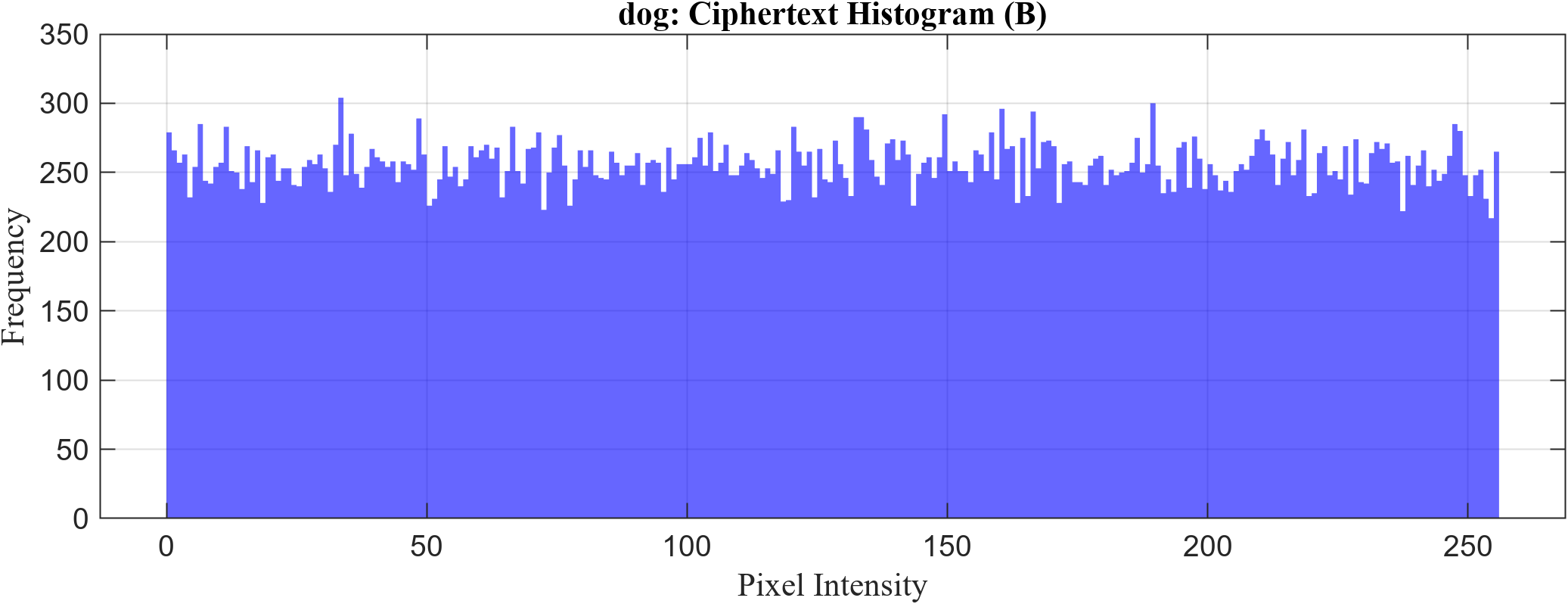}
\end{minipage}
}

\caption{QT-SVD histogram results at frame~25. For each video, the panels
from top to bottom show the overlaid histograms of the original RGB channels,
followed by the histograms of the encrypted-core R, G, and B components. Each histogram uses 256 intensity bins.}
\label{fig:histograms}
\end{figure}

\begin{figure}[!htbp]
\centering

% ---------- First row: three videos ----------
\begin{minipage}[t]{0.3\linewidth}
\centering
{\scriptsize blackswan}\\[3pt]
\includegraphics[width=\linewidth]{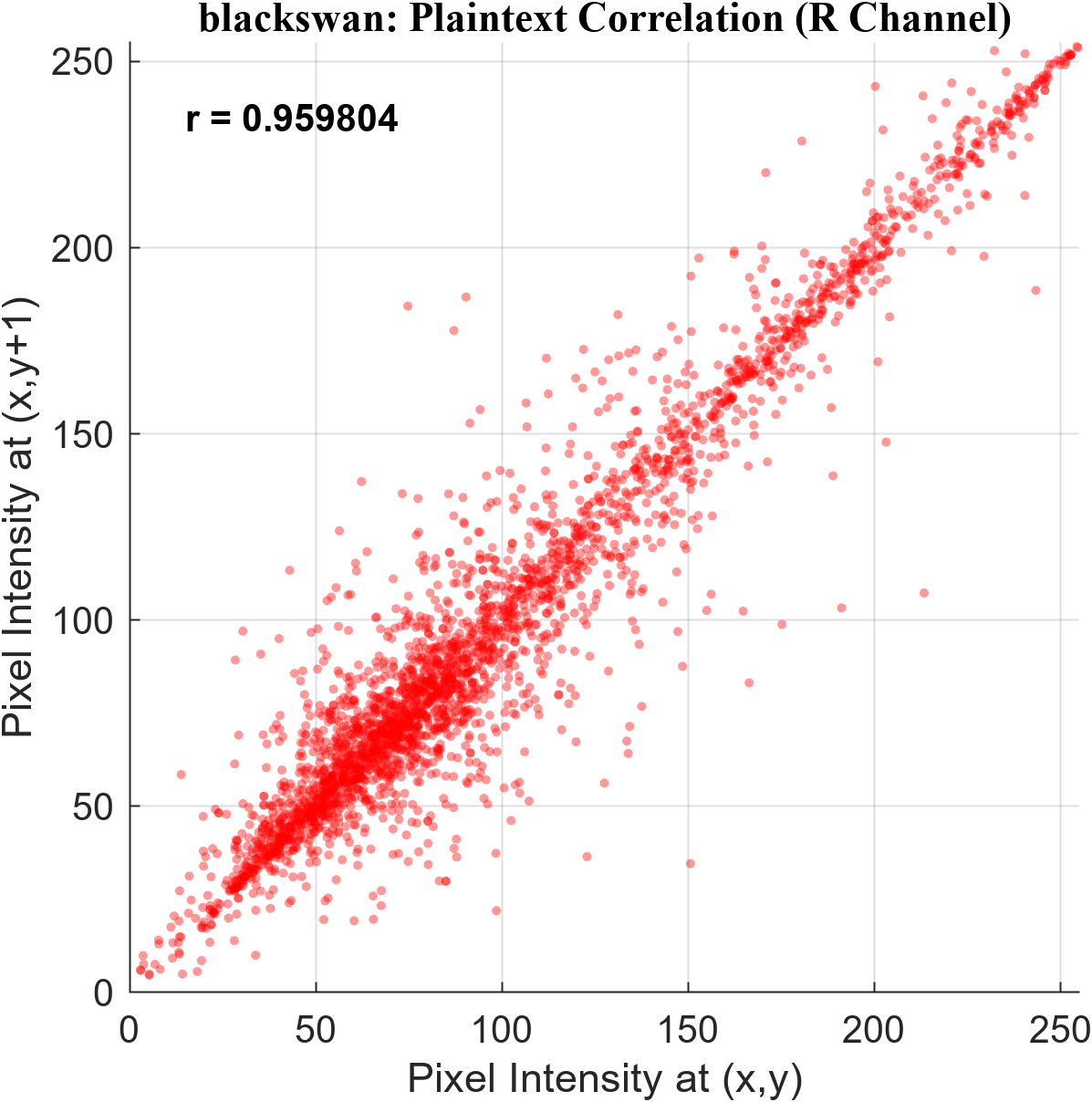}\\[-1pt]
\includegraphics[width=\linewidth]{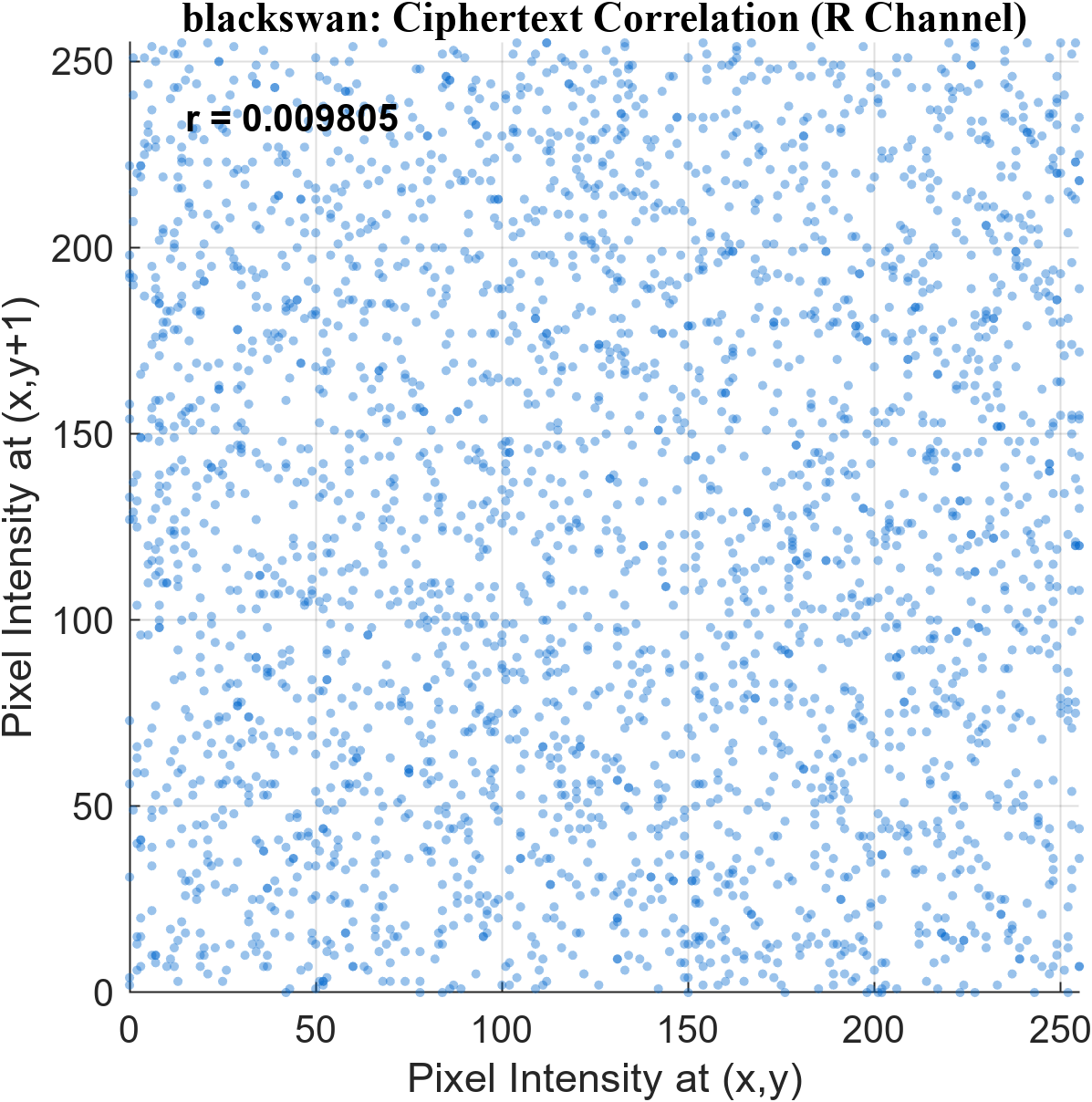}
\end{minipage}
\hfill
\begin{minipage}[t]{0.3\linewidth}
\centering
{\scriptsize bmx-trees}\\[3pt]
\includegraphics[width=\linewidth]{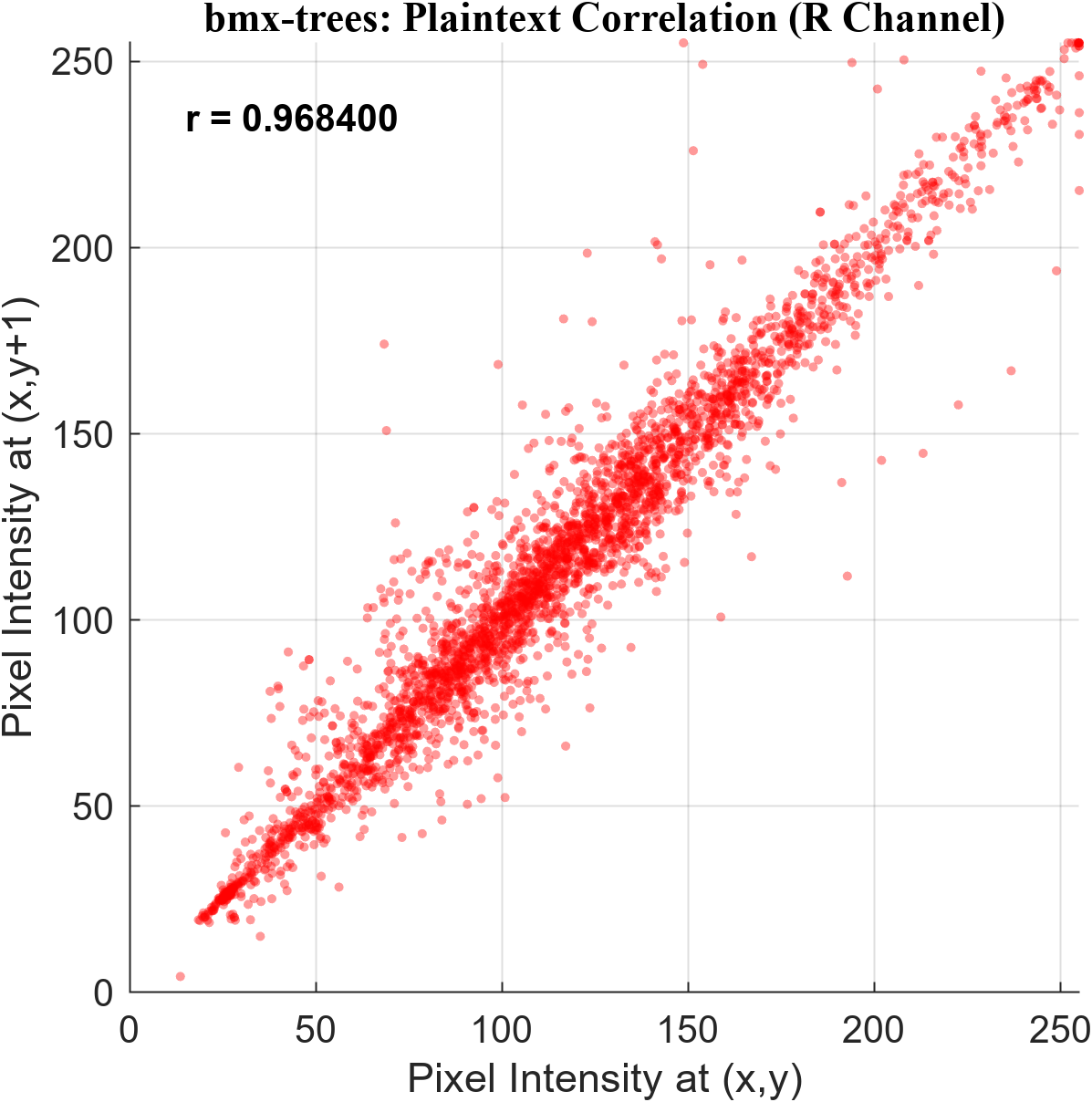}\\[-1pt]
\includegraphics[width=\linewidth]{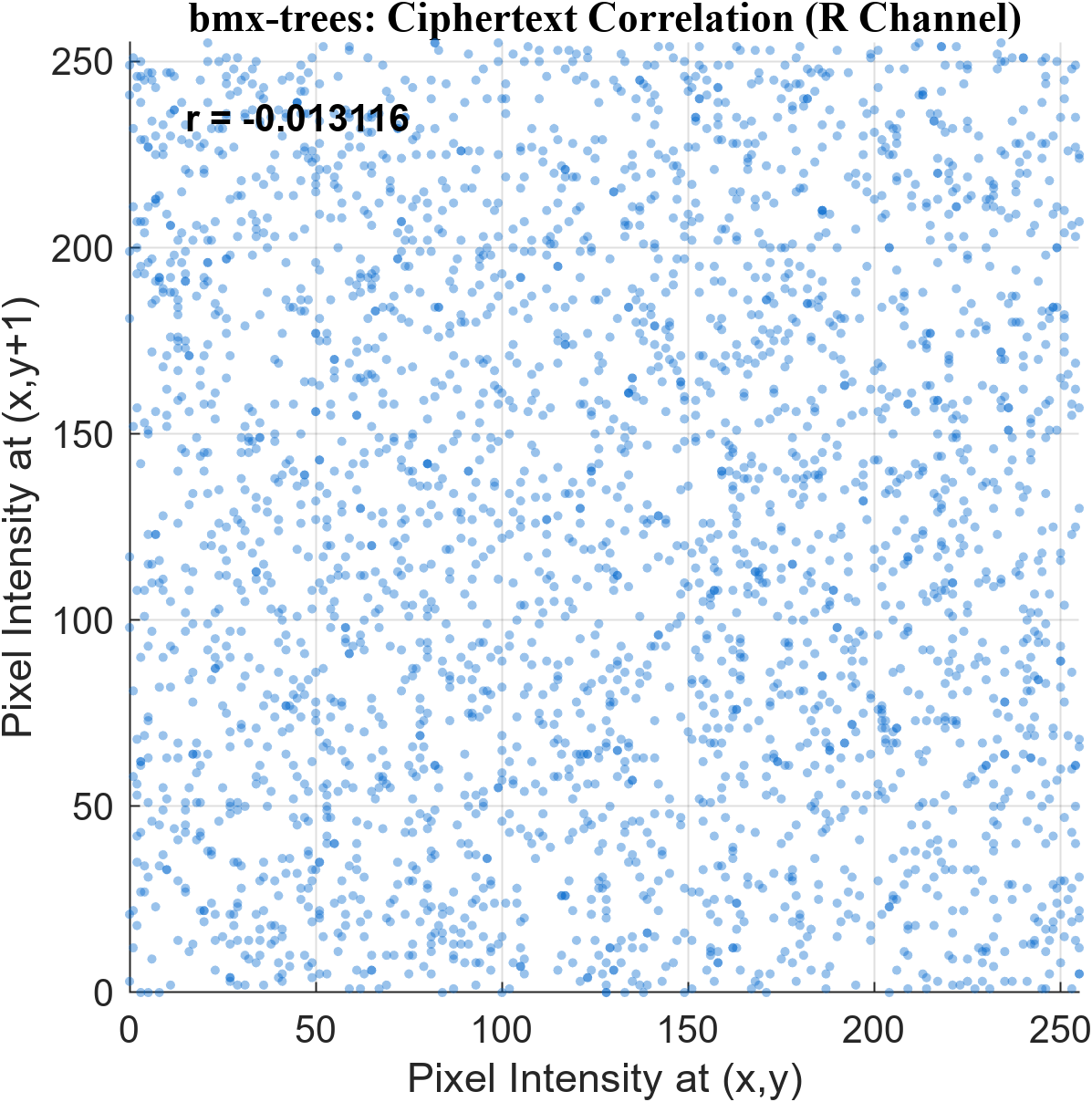}
\end{minipage}
\hfill
\begin{minipage}[t]{0.3\linewidth}
\centering
{\scriptsize car-roundabout}\\[3pt]
\includegraphics[width=\linewidth]{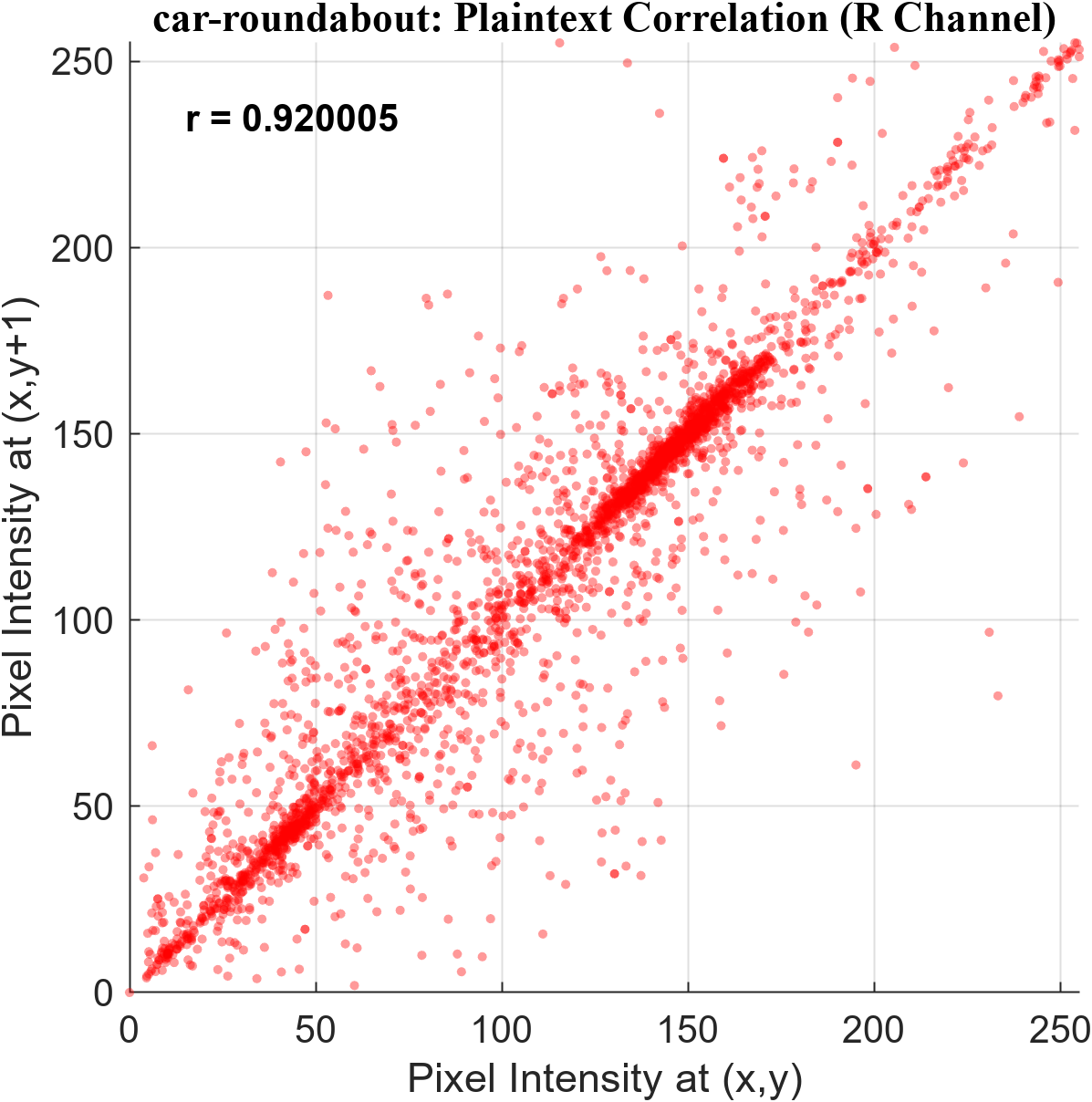}\\[-1pt]
\includegraphics[width=\linewidth]{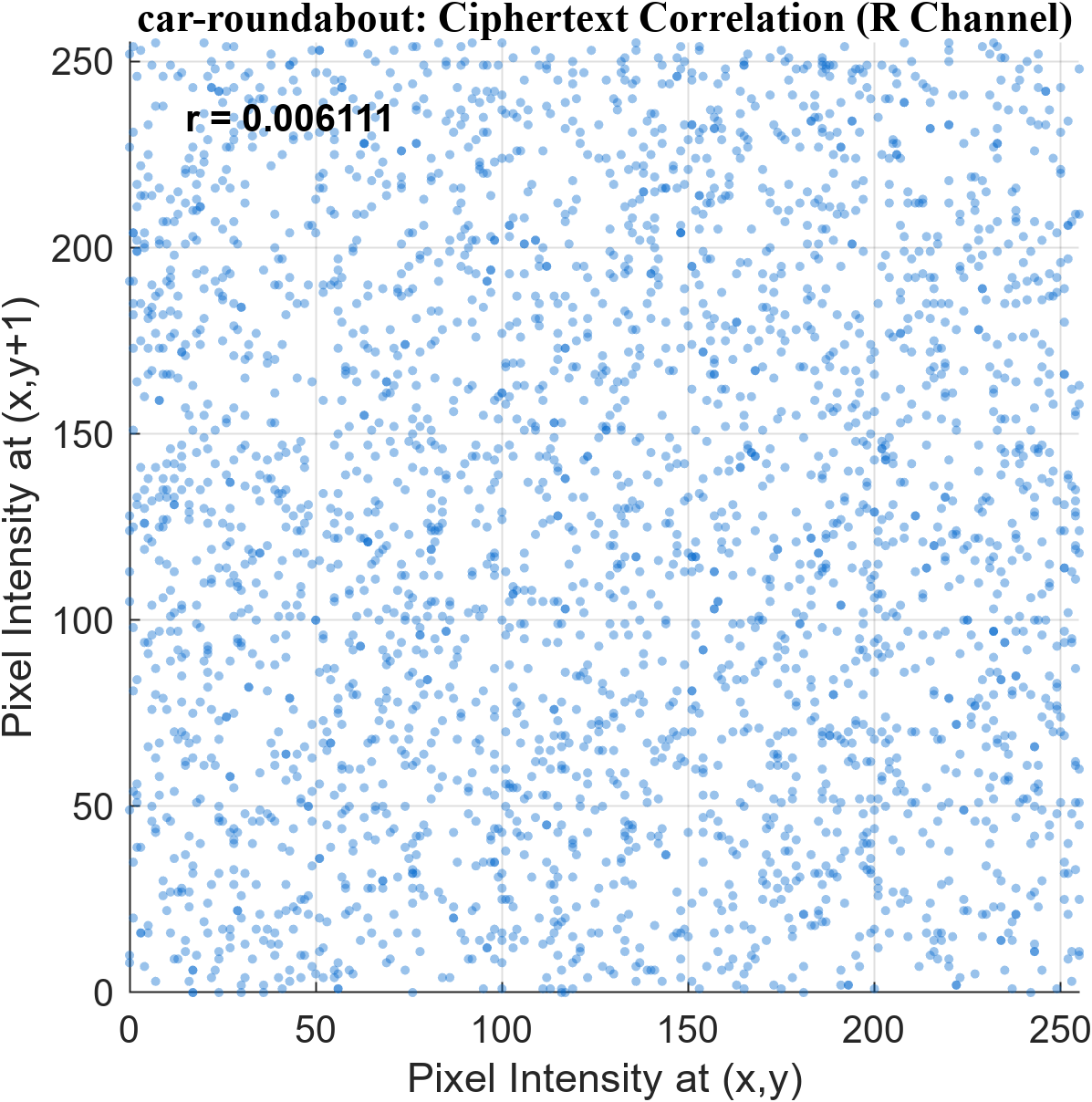}
\end{minipage}

\vspace{6pt}

% ---------- Second row: two videos ----------
\makebox[\linewidth][c]{%
\begin{minipage}[t]{0.3\linewidth}
\centering
{\scriptsize dance-twirl}\\[3pt]
\includegraphics[width=\linewidth]{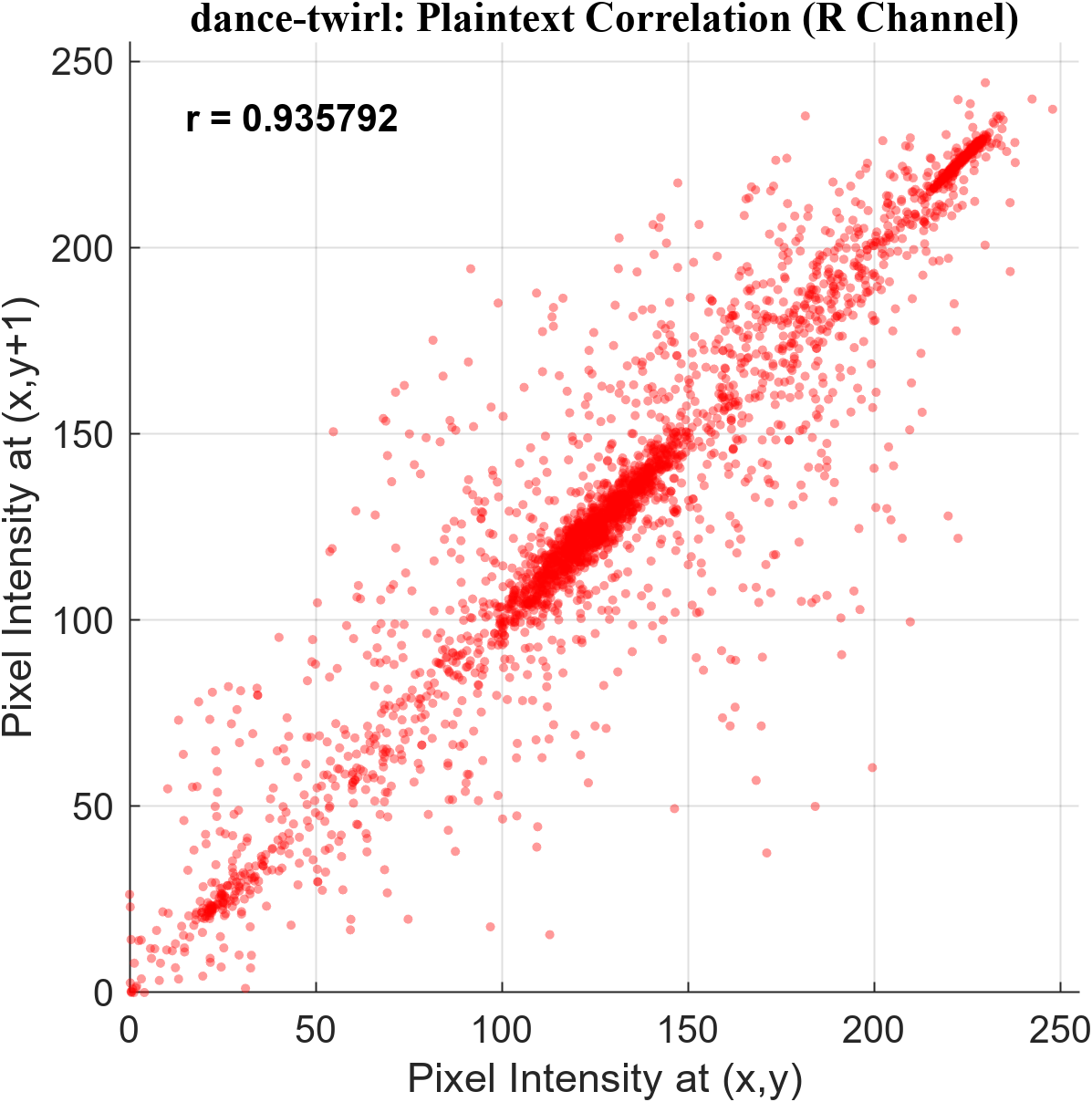}\\[-1pt]
\includegraphics[width=\linewidth]{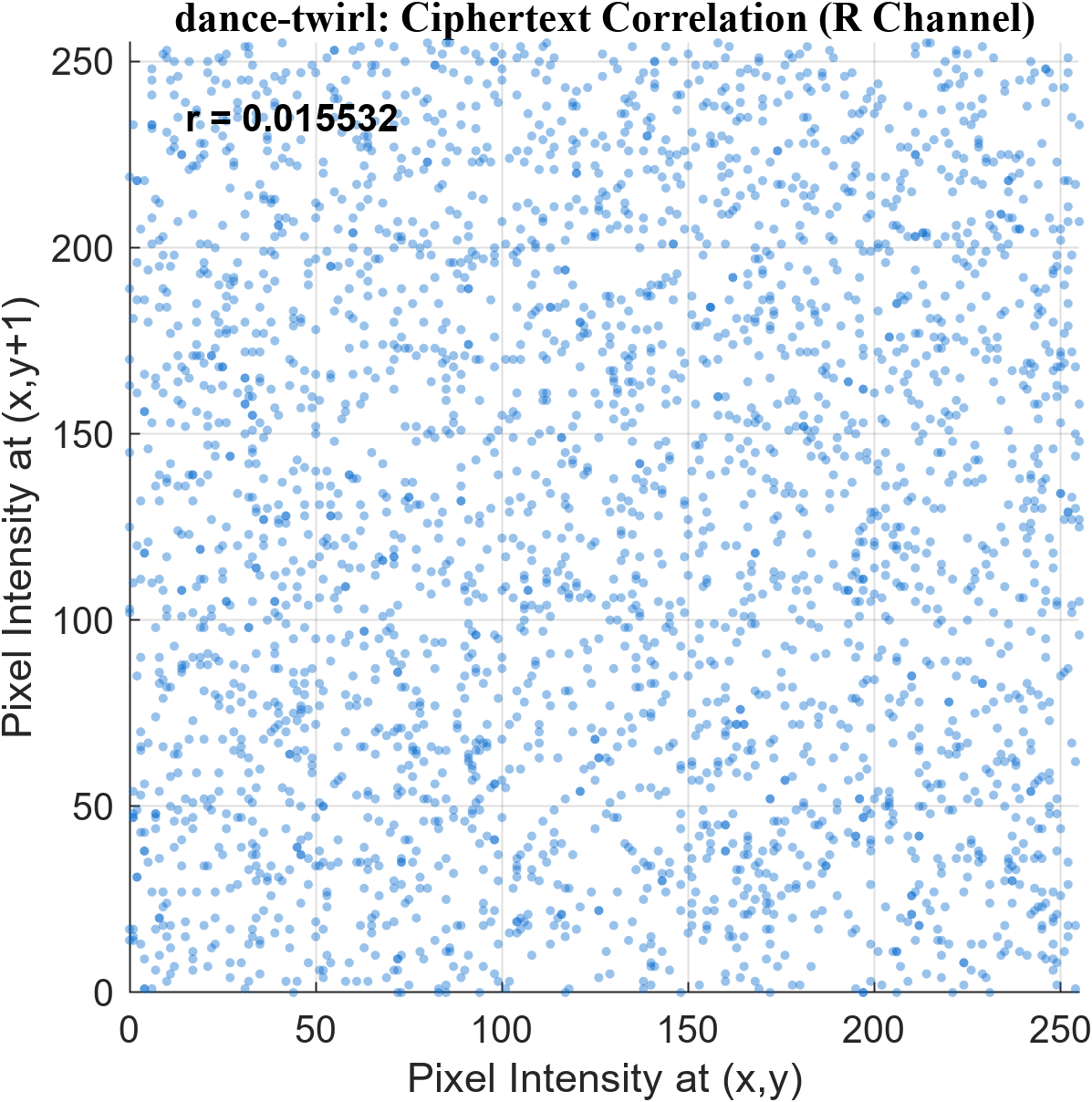}
\end{minipage}
\hspace{0.03\linewidth}
\begin{minipage}[t]{0.3\linewidth}
\centering
{\scriptsize dog}\\[3pt]
\includegraphics[width=\linewidth]{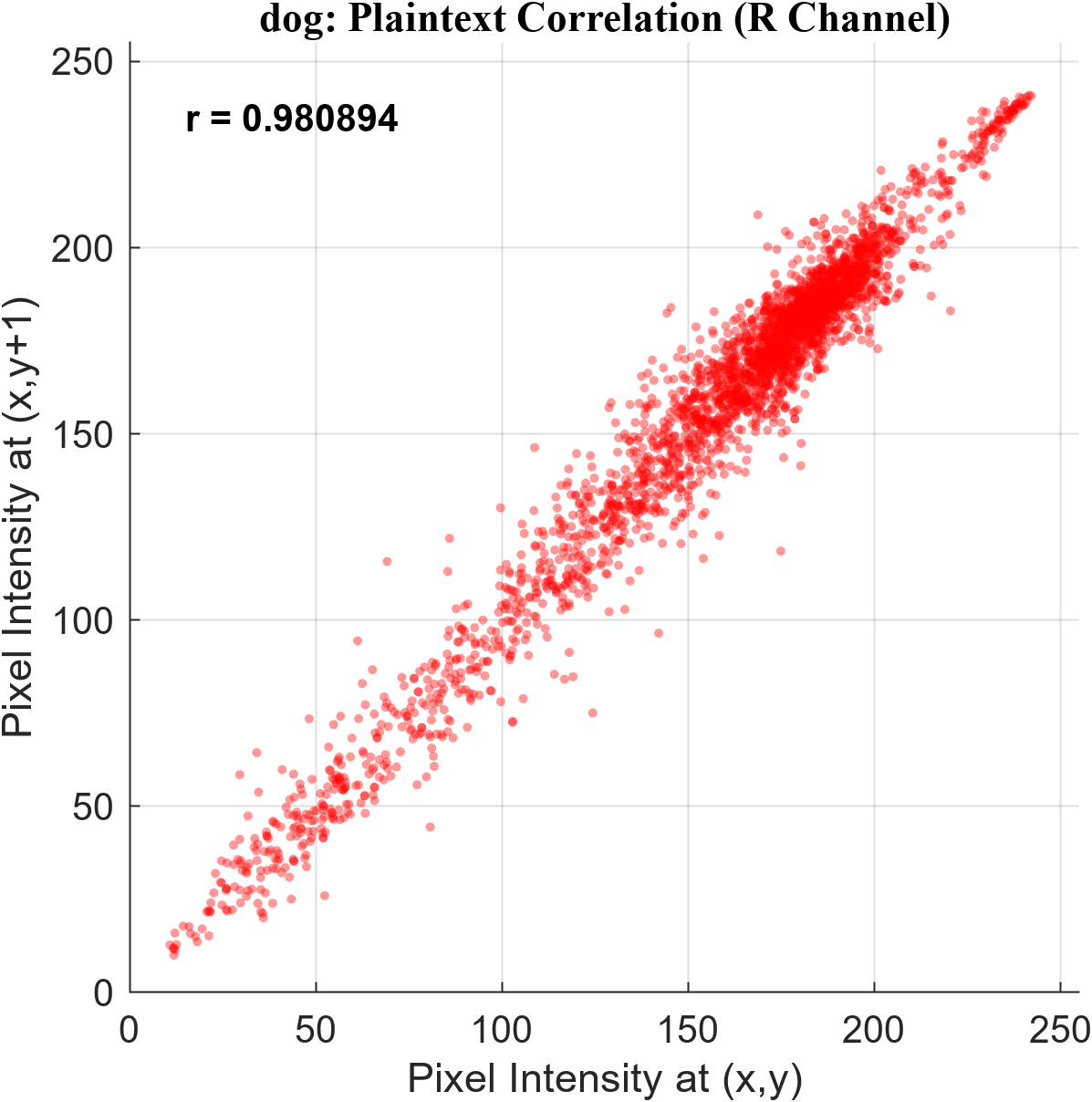}\\[-1pt]
\includegraphics[width=\linewidth]{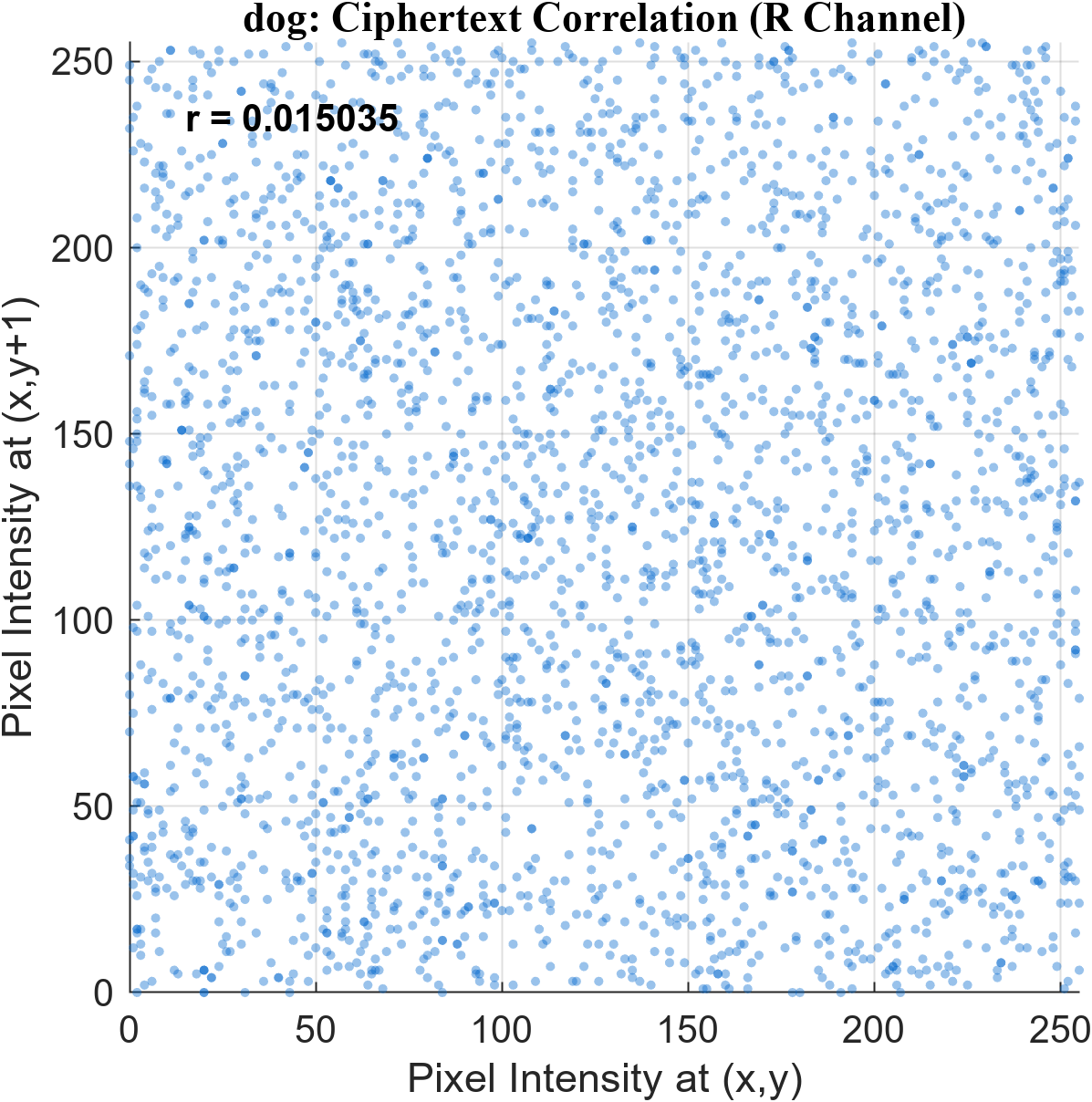}
\end{minipage}
}

\caption{QT-SVD adjacent-pixel correlation plots at frame 25 (R channel).}
\label{fig:correlations}
\end{figure}

Table~\ref{tab:security} compares the two methods quantitatively, where
$r_{\rm cipher}$ denotes the Pearson correlation coefficient computed from
3,000 horizontally adjacent pixel pairs in the encrypted R channel at
frame~25. The five-video mean absolute values of $r_{\rm cipher}$ are
$0.0119$ for QT-SVD and $0.0127$ for T-SVD
(Table~\ref{tab:securitysummary}). Neither method yields a smaller absolute
correlation for all five videos, indicating no consistent advantage of one
method over the other in terms of the measured horizontal adjacent-pixel
correlation.
\begin{table}[!htbp]
\centering\small
\caption{Statistical comparison for the five videos.}
\label{tab:security}
\setlength{\tabcolsep}{4pt}
\begin{tabular}{lrrrrrr}
\toprule
Video
& \multicolumn{2}{c}{NPCR (\%)}
& \multicolumn{2}{c}{UACI (\%)}
& \multicolumn{2}{c}{$r_{\rm cipher}$} \\
\cmidrule(lr){2-3}
\cmidrule(lr){4-5}
\cmidrule(lr){6-7}
& QT-SVD & T-SVD
& QT-SVD & T-SVD
& QT-SVD & T-SVD \\
\midrule
blackswan
& 99.6069 & 99.6144
& 33.4454 & 33.4740
& 0.0098 & 0.0257 \\

bmx-trees
& 99.6038 & 99.6093
& 33.4467 & 33.4511
& -0.0131 & -0.0004 \\

car-roundabout
& 99.6087 & 99.6104
& 33.4822 & 33.4612
& 0.0169 & 0.0250 \\

dance-twirl
& 99.6118 & 99.6107
& 33.4583 & 33.4546
& 0.0155 & 0.0038 \\

dog
& 99.6103 & 99.6088
& 33.4796 & 33.4631
& 0.0150 & 0.0037 \\
\bottomrule
\end{tabular}
\end{table}

\begin{table}[!htbp]
\centering\small
\caption{Equal-video summary of ciphertext statistics.}
\label{tab:securitysummary}
\setlength{\tabcolsep}{4pt}
\begin{tabular}{lrrr}
\toprule
Metric & QT-SVD & T-SVD & Random reference \\
\midrule
NPCR (\%) & 99.6083 & 99.6107 & 99.6094 \\
UACI (\%) & 33.4624 & 33.4608 & 33.4635 \\
Mean $|r_{\rm cipher}|$ & 0.0119 & 0.0127 & 0 \\
\bottomrule
\end{tabular}
\end{table}

\subsubsection{Differential Sensitivity}
As shown in Table~\ref{tab:securitysummary}, both methods yield mean NPCR and
UACI values close to the theoretical reference values of $99.6094\%$ and
$33.4635\%$, respectively. This indicates that a small plaintext perturbation produces widespread changes in the encrypted core, with comparable differential sensitivity for QT-SVD and T-SVD.

\subsection{Stability under Pre-Encryption Input Perturbations}
Three perturbations are applied to the clean RGB videos before decomposition: independent salt-and-pepper noise of density $0.01$, zero-mean Gaussian noise of variance $0.01$ in normalized intensity units, and a centered $12\times12$ zero occlusion in every frame and color channel. At the tested resolution the occlusion occupies rows 123--134 and columns 123--134. The random generator is reset to seed 20260716 for paired method comparisons, producing identical perturbed inputs.

Each perturbed video is decomposed afresh, its core is encrypted and recovered, and the video is reconstructed using its corresponding factors. Both the perturbed input and reconstructed output are compared against the same clean video. With $P$ denoting PSNR and $S$ denoting SSIM, define $\Delta P=P_{\rm out}-P_{\rm in}$ and $\Delta S=S_{\rm out}-S_{\rm in}$. This paired design isolates the additional quality change introduced by the processing pipeline from the distortion already present in its input.
\begin{table}[!htbp]
\centering\small
\caption{Paired pre-encryption perturbation results, averaged equally over the five videos.}
\label{tab:robustness}
\setlength{\tabcolsep}{4pt}
\begin{tabular}{llrrrrr}
\toprule
Perturbation & Method & $P_{\rm in}$ & $P_{\rm out}$ & $\Delta P$ & $S_{\rm in}$ & $S_{\rm out}$ \\
\midrule
Salt-and-pepper & QT-SVD & 25.2034 & 25.2067 & +0.0033 & 0.8295 & 0.8294 \\
Salt-and-pepper & T-SVD & 25.2034 & 25.1998 & -0.0036& 0.8295 & 0.8268 \\
\midrule
Gaussian & QT-SVD & 20.2122 & 20.2128 & +0.0006 & 0.4223 & 0.4223 \\
Gaussian & T-SVD & 20.2122 & 20.2103 & -0.0019 & 0.4223 & 0.4222 \\
\midrule
Occlusion & QT-SVD & 33.2326 & 33.2299 & -0.0027 & 0.9961 & 0.9959 \\
Occlusion & T-SVD & 33.2326 & 33.0139 & -0.2187 & 0.9961 & 0.9894 \\
\bottomrule
\end{tabular}
\par\smallskip\begin{minipage}{\linewidth}\footnotesize $P$ denotes PSNR (dB), $S$ denotes SSIM, and $\Delta P=P_{\rm out}-P_{\rm in}$. Both input and output are compared with the same clean reference video. Identical perturbed inputs are used by the two methods.\end{minipage}
\end{table}

For QT-SVD, the mean PSNR changes are $+0.003292$, $+0.000558$, and $-0.002790$~dB for salt-and-pepper noise, Gaussian noise, and occlusion, respectively. Its corresponding mean SSIM changes are $-0.000094$, $+0.000015$, and $-0.000205$. Thus, the additional average distortion is small under each tested input perturbation. The slight positive PSNR changes do not by themselves demonstrate a denoising capability.

The two methods show little change under salt-and-pepper and Gaussian
perturbations, whereas a clearer difference emerges under occlusion.
As shown in Table~\ref{tab:robustness}, under occlusion, the mean PSNR
decreases by $0.0027$ dB for QT-SVD and $0.2187$ dB for T-SVD, while
the corresponding mean SSIM decreases are $0.0002$ and $0.0067$,
respectively. Thus, QT-SVD exhibits smaller degradation than T-SVD
under the tested occlusion perturbation.

These experiments concern perturbed plaintext inputs. The ciphertext is not corrupted during transmission or storage, and the factors are recomputed after perturbation. Accordingly, the results measure stability of processing already disturbed videos; they do not establish error tolerance of encrypted data or recovery of missing scene content.

% BEGIN QT-SVD IMPLEMENTATION COMPARISON

\subsection{Decomposition Time Comparison}
Table~\ref{tab:timing} compares the decomposition times of QT-SVD and T-SVD for the five video sequences. QT-SVD requires less decomposition time for every tested sequence. The equal-video mean decomposition time is 16.189~s for QT-SVD and 37.352s for T-SVD, corresponding to a reduction of approximately 56.7\%. These results demonstrate the computational efficiency of the QT-SVD implementation in the tested video decomposition tasks.

\begin{table}[htbp]
\centering
\caption{Decomposition times for the five video sequences.}
\label{tab:timing}
\begin{tabular}{lrr}
\toprule
Video & QT-SVD (s) & T-SVD (s) \\
\midrule
blackswan      & 10.456  & 14.864 \\
bmx-trees      & 18.738 & 51.945 \\
car-roundabout & 17.099 & 39.765 \\
dance-twirl    & 22.442 & 53.298 \\
dog            & 12.211 & 37.352 \\
\midrule
Five-video mean & 16.189 & 37.352 \\
\bottomrule
\end{tabular}
\end{table}

\subsection{Discussion and Limitations}
The five-sequence study supports three conclusions under the stated settings. First, both QT-SVD and T-SVD accurately reconstruct the original videos from their decomposition factors. Second, QT-SVD yields higher PSNR and SSIM after decryption in every tested video. Third, both pipelines produce similar aggregate ciphertext sensitivity and low measured R-channel adjacency correlation, while QT-SVD introduces less additional average degradation for the tested occluded inputs.

The scope is limited to five sequences resized to $256\times256$, one set of masking/permutation parameters, and the retained auxiliary information used by the prototype. No claim is made about lossless recovery of the floating-point video, rank-compression performance, ciphertext corruption tolerance, or formal cryptographic security. Broader evaluation would require additional sequences and resolutions, controlled repeated timing, and an explicit key and auxiliary-data protection model. The present evidence identifies QT-SVD as a useful representation for this video recovery pipeline and motivates those further tests.

\section*{Funding}
\begin{itemize}
\item The first author was supported by the National Natural Science Foundation of China (NSFC) (No. 11901079), and China Postdoctoral Science Foundation (No. 2021M700751), and the Scientific and Technological Research Program Foundation of Jilin Province (No. JJKH20190690KJ; No. JJKH20220091KJ; No. JJKH20250851KJ).
% \item The fourth author is supported by the Ministry of Science, Technological Development and Innovations, Republic of Serbia (No. 451-03-137/2025-03/200124).
\end{itemize}

\section*{Conflict of Interest}
The authors declare that they have no potential conflict of interest.

\section*{Data Availability}
Data sharing is not applicable to this article, as no datasets were generated or analyzed during the current study.

\section*{Acknowledgments}

The authors are grateful to Michael K. Ng of Hong Kong Baptist University for his valuable suggestions, which have helped improve the quality of this paper.

\bibliographystyle{abbrv}
%\bibliography{mybibfile}

\end{document}